\documentclass{amsart}\def\distribute{}
\date{July 10, 2018}
\usepackage{graphicx}
\usepackage[dvipsnames]{xcolor}
\title[
   The analytic extension of 
  exceptional catenoids]{
  Analytic extension of 
  exceptional\\ constant mean curvature
  one catenoids \\
  in de Sitter 3-space
}

\author{S. Fujimori}
\author{Y. Kawakami}
\author{M. Kokubu}
\author{W. Rossman}
\author{M. Umehara}
\author{K. Yamada}

\address[Shoichi Fujimori]{
  Department of Mathematics, Okayama University, 
  Okayama 700-8530, Japan.
}
\email{%
  fujimori@math.okayama-u.ac.jp
}

\address[Yu Kawakami]{
  Graduate School of Natural Science and Technology,
  Kanazawa University,
  Kanazawa, 920-1192, Japan.}
\email{%
  y-kwkami@se.kanazawa-u.ac.jp
}

\address[Masatoshi Kokubu]{
  Department of Mathematics, School of Engineering, 
  Tokyo Denki University, 
  Tokyo 120-8551, Japan.}
\email{%
  kokubu@cck.dendai.ac.jp
}

\address[Wayne Rossman]{
   Department of Mathematics,
   Faculty of Science,
   Kobe University,
   Rokko, Kobe 657-8501, Japan.}
\email{%
 wayne@math.kobe-u.ac.jp
}

\address[Masaaki Umehara]{
  Department of Mathematical and Computing Sciences,
  Tokyo Institute of Technology,
  2-12-1-W8-34, O-okayama Meguro-ku,
  Tokyo 152-8552, Japan.
}
\email{%
  umehara@is.titech.ac.jp}

\address[Kotaro Yamada]{
  Department of Mathematics,
  Tokyo Institute of Technology,
  Tokyo 152-8551, Japan.
}
\email{kotaro@math.titech.ac.jp}


\newcommand{\op}[1]{{\operatorname{#1}}}
\newcommand{\SL}{\op{SL}}
\newcommand{\pmt}[1]{{\begin{pmatrix} #1  \end{pmatrix}}}
\newcommand{\R}{\boldsymbol{R}}
\newcommand{\C}{\boldsymbol{C}}
\newcommand{\Z}{\boldsymbol{Z}}
\renewcommand{\phi}{\varphi}
\newcommand\A{{\mathcal A}}
\newcommand\CC{{\mathcal C}}
\newcommand\D{{\mathcal D}}
\newcommand{\inner}[2]{\left\langle{#1},{#2}\right\rangle}
\newcommand{\trans}[1]{\vphantom{#1}^{t}#1}
\newcommand{\Herm}{\op{Herm}}
\newcommand{\imag}{\op{i}}
\newcommand{\sgn}{\op{sgn}}
\newcommand{\I}{\op{I}}
\newcommand{\II}{\op{I\!I}}
\newcommand{\J}{\op{J}}
\newtheorem{thm}{Theorem}

\theoremstyle{definition}

\newtheorem{rem}{Remark}

\newtheorem{prop}[thm]{Proposition}
\newtheorem*{ack}{Acknowledgements}
\begin{document}
\maketitle

\begin{abstract}
Catenoids in de Sitter $3$-space $S^3_1$ belong to 
a certain class of space-like constant mean curvature 
one surfaces.
In a previous work,
the authors classified such catenoids, 
and found that two different classes of countably many
exceptional elliptic catenoids are not realized
as closed subsets in $S^3_1$.
Here we show that such exceptional 
catenoids 
have closed analytic extensions
in $S^3_1$ with interesting properties.
\end{abstract}

\section{Introduction.}
We denote by $S^3_1$ the de Sitter $3$-space,
which is a simply-connected Lorentzian 
$3$-manifold with constant sectional curvature $1$.
Let $\R^4_1$ be the Lorentz-Minkowski $4$-space
with the metric $\inner{~}{~}$ of 
signature $(-+++)$.
Then  
\[
  S^3_1=
  \{X \in\R^4_1\, ; \, \inner{X}{X}=1\}
\]
with metric induced from $\R^4_1$. 
We identify $\R^4_1$ with the  
$2\times 2$ Hermitian matrices
$\Herm(2)$ by
\[
(t,x,y,z)
     \longleftrightarrow
     \pmt{
      t+z       & x+\imag y \\
      x-\imag y & t-z
       },
\]
where $\imag=\sqrt{-1}$.
Then $S^3_1$ is represented as 
\begin{align*}
  S^3_1&=\{X\in\Herm(2)\,;\,\det X=-1\}
       =\{ae_3a^*\,;\,a\in \SL(2,\C)\},
\end{align*}
where $a^*:=\trans{\overline{a}}$ is 
the conjugate transpose of $a$, 
and 
\[
   e_3:=\begin{pmatrix}
	       1 & \hphantom{-}0\\
	       0 & -1
	\end{pmatrix}.
\]

To draw surfaces in $S^3_1$,
we use the \emph{stereographic hollow ball model} given in \cite{FNSSY} 
as follows:
\begin{multline}\label{eq:hollow}
   \varPi :
     S^3_1 \ni (t,x,y,z)\longmapsto \frac{1}{\delta}(x,y,z)\in \R^3\\
      \left(\delta:= t+\sqrt{t^2+x^2+y^2+z^2}=t+\sqrt{2t^2+1}\right).
\end{multline}
This projection $\varPi$ is the composition
of central projection of $S^3_1$
to the unit sphere $S^3$ centered at the
origin in $\R^4$ and
usual stereographic projection of
$S^3$ into $\R^3$ from $(0,0,0,-1)$.
The image of $\varPi$
is the set 
\begin{equation}\label{eq:model}
 \D^3:=
    \left\{
	     \xi\in\R^3\,
	     ;\,\sqrt{2}-1<|\mathbf \xi|<\sqrt{2}+1
			       	    \right\},
\end{equation}
where $|\xi|:=\sqrt{\xi_1^2+\xi_2^2+\xi_3^2}$
for $\xi=(\xi_1,\xi_2,\xi_3)$.

In \cite{FKKRUY2}, the authors classified
all \emph{catenoids} in $S^3_1$
(i.e.\  weakly complete constant mean curvature one
surfaces in $S^3_1$ of genus zero
with two regular ends  whose hyperbolic Gauss map is
of degree one).
There are three types of catenoids:
\begin{itemize}
 \item elliptic catenoids,
 \item the parabolic catenoid, and
 \item hyperbolic catenoids.
\end{itemize}
Parabolic catenoids have only one
congruence class, whose secondary Gauss map
is given by
$$
g=\frac{1+\log z}{-1+\log z},
$$
and they are rotationally symmetric surfaces
with one cone-like singular point
and two embedded ends. 
On the other hand, the
secondary Gauss map of hyperbolic catenoids
are of the form
$$
g=\frac{g_0-\imag}{g_0+\imag}, \qquad
g_0:=\exp((m+\imag \tau)\log z)=z^{m+\imag\tau},
$$
where $m$ is a non-negative integer, and
$\tau$ is a non-zero real number.
When $m\ne 0$ (resp. $m=0$), 
hyperbolic catenoids admit only
cuspidal edge singularities
(resp. cone-like singular points),
see \cite[Page 36]{FKKRUY2}.
Recently, in a joint work with Seong-Deog Yang, 
the authors \cite{FKKRUYY2} proved 
that
all hyperbolic catenoids do not admit
any analytic extension.

On the other hand, there are 
many subclasses of
elliptic catenoids,
whose secondary Gauss maps
$g$ are given by
\begingroup
\renewcommand{\theenumi}{(\roman{enumi})}
\renewcommand{\labelenumi}{(\roman{enumi})}
\begin{enumerate}
 \item\label{item:cat:1} 
      $g=z^\alpha$\,\, ($0<\alpha<1$),
 \item\label{item:cat:2} 
      $g=z^\alpha$\,\, ($\alpha>1$),
 \item\label{item:cat:3} 
      $g=z^m+c$\,\, ($m=2,3,\dots$) with $c\in (0,\infty)\setminus\{1\}$,
 \item\label{item:cat:4} 
      $g=z^m+1$\,\, ($m=2,3,\dots$),
 \item\label{item:cat:5} 
      $g=(z^m-1)/(z^m+1)$\,\,
      ($m=2,3,\dots$).
\end{enumerate}
\endgroup
Except for the two cases \ref{item:cat:4} and \ref{item:cat:5},
all elliptic catenoids are closed subsets
of $S^3_1$, since the singular sets
of catenoids of type \ref{item:cat:1}--\ref{item:cat:3} are compact.
In this paper, we call
the catenoids in the class \ref{item:cat:4} (resp.\ \ref{item:cat:5})
\emph{exceptional catenoids of type I}
(resp.\ \emph{exceptional catenoids of type II})
and we study these two classes.

For each $m=2$, $3$, \dots, we set
\begin{equation}\label{eq:ec0}
   F^{\I}_m
   :=
\ifx\distribute\undefined
   \text{\small$
\fi
   \dfrac{z^{-\frac{m+1}{2}}}{2 \sqrt{m}}
\ifx\distribute\undefined
   $}
   \text{\footnotesize$
\fi
   \pmt{
      (m+1)z &&   z\bigl((m-1) z^{m} -m-1)\bigr) \\
      m-1  &&   (m+1) z^{m}  -m+1
       }
\ifx\distribute\undefined
$}
\fi
\end{equation}
and
\ifx\distribute\undefined
\begin{multline}\label{eq:ec}
 F^{\II}_m
    :=
   \text{\small$
 \dfrac{z^{-\frac{m+1}{2}}}{2 \sqrt{2m}}
    \times
$}\\
 \text{\footnotesize$
 \pmt{
   z\bigl((1-m)z^m+m+1\bigr) &&
   z\bigl((m-1)z^m+m+1\bigr) \\
   -(m+1)z^m+m-1 &&
    (m+1)z^m+m-1
  }.
$}
\end{multline}
\else
\begin{equation}\label{eq:ec}
 F^{\II}_m
    :=
 \dfrac{z^{-\frac{m+1}{2}}}{2 \sqrt{2m}}
 \pmt{
   z\bigl((1-m)z^m+m+1\bigr) &&
   z\bigl((m-1)z^m+m+1\bigr) \\
   -(m+1)z^m+m-1 &&
    (m+1)z^m+m-1
  }.
\end{equation}
\fi
The maps
$f^{\J}_m:\C\setminus \{0\}\to S^3_1$
defined by
\[
   f^{\J}_m
     :=F^{\J}_m e_3
     (F_m^{\J})^*
      \qquad (\J=\I, \II)
\]
give the exceptional catenoids.
These expressions
are obtained by shifting $m$ to $m-1$
in \cite[Prop.\ 4.9]{FKKRUY2}.
We will show that the image of each $f^{\J}_m$ ($\J=\I,\II$)
has an analytic 
extension $\CC^{\J}_m$ 
which is a closed set in $S^3_1$.

A subset $\A$ of a manifold $M^n$
is called \emph{almost embedded}
(resp.\ \emph{almost immersed})
if there is a discrete subset $D$ of $\A$
such that $\A\setminus D$ is 
the image of an embedding (resp.\ an immersion)
of a manifold into $M^n$.
For example (cf.\ \cite{FKKRUY2}), 
\begin{itemize}
\item catenoids of class \ref{item:cat:3}
      are not almost immersed,
\item catenoids of class \ref{item:cat:1}
      are almost immersed, but not
      almost embedded,
\item catenoids of class \ref{item:cat:2}
      are almost embedded.
\end{itemize}
\ifx\undefined\distribute
\begin{figure}[htb]%
\else
\begin{figure}
\fi
 \begin{center}
       \includegraphics[width=3.6cm]{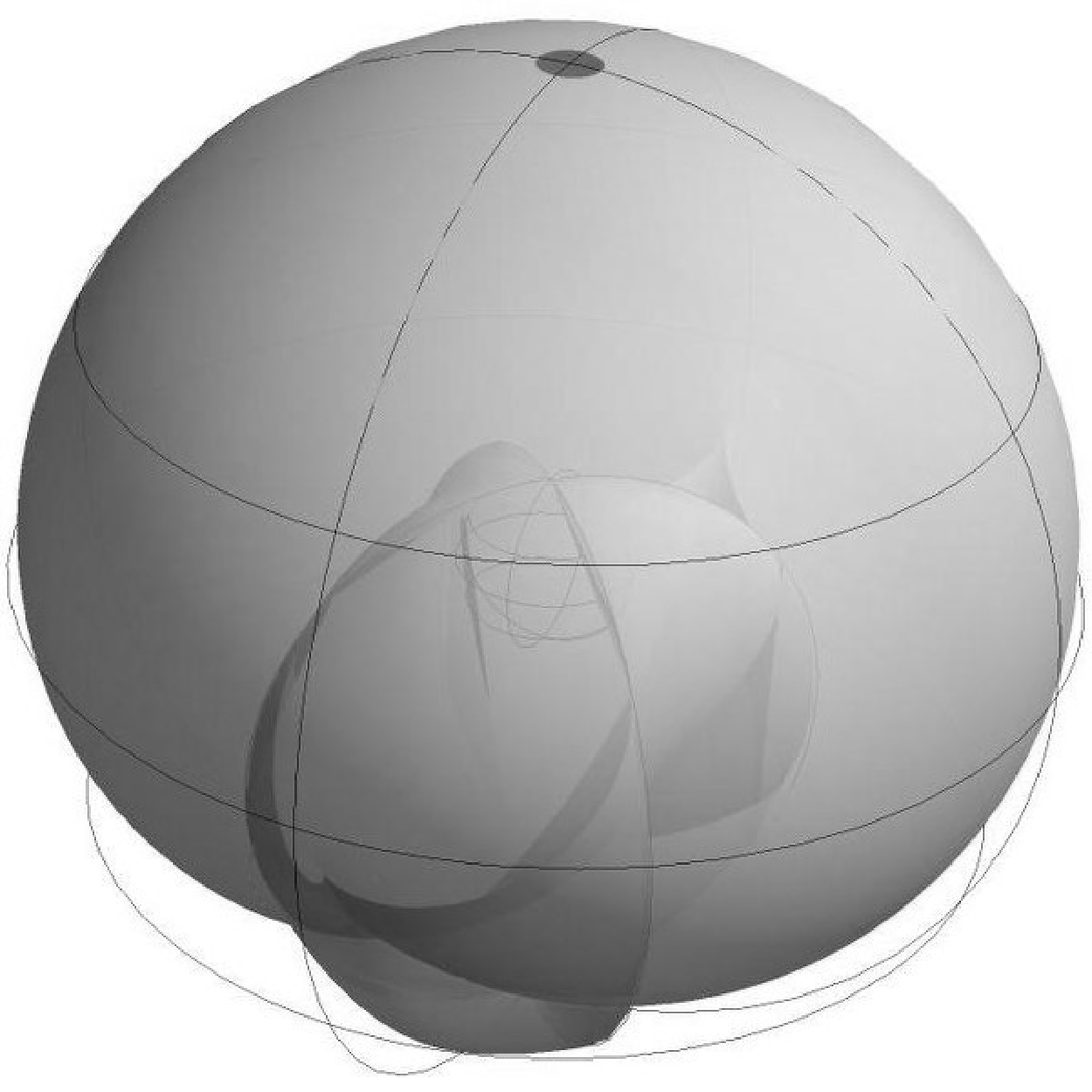} 
\ifx\undefined\distribute
\else
\hspace{0.5cm}
\fi
       \includegraphics[width=3.6cm]{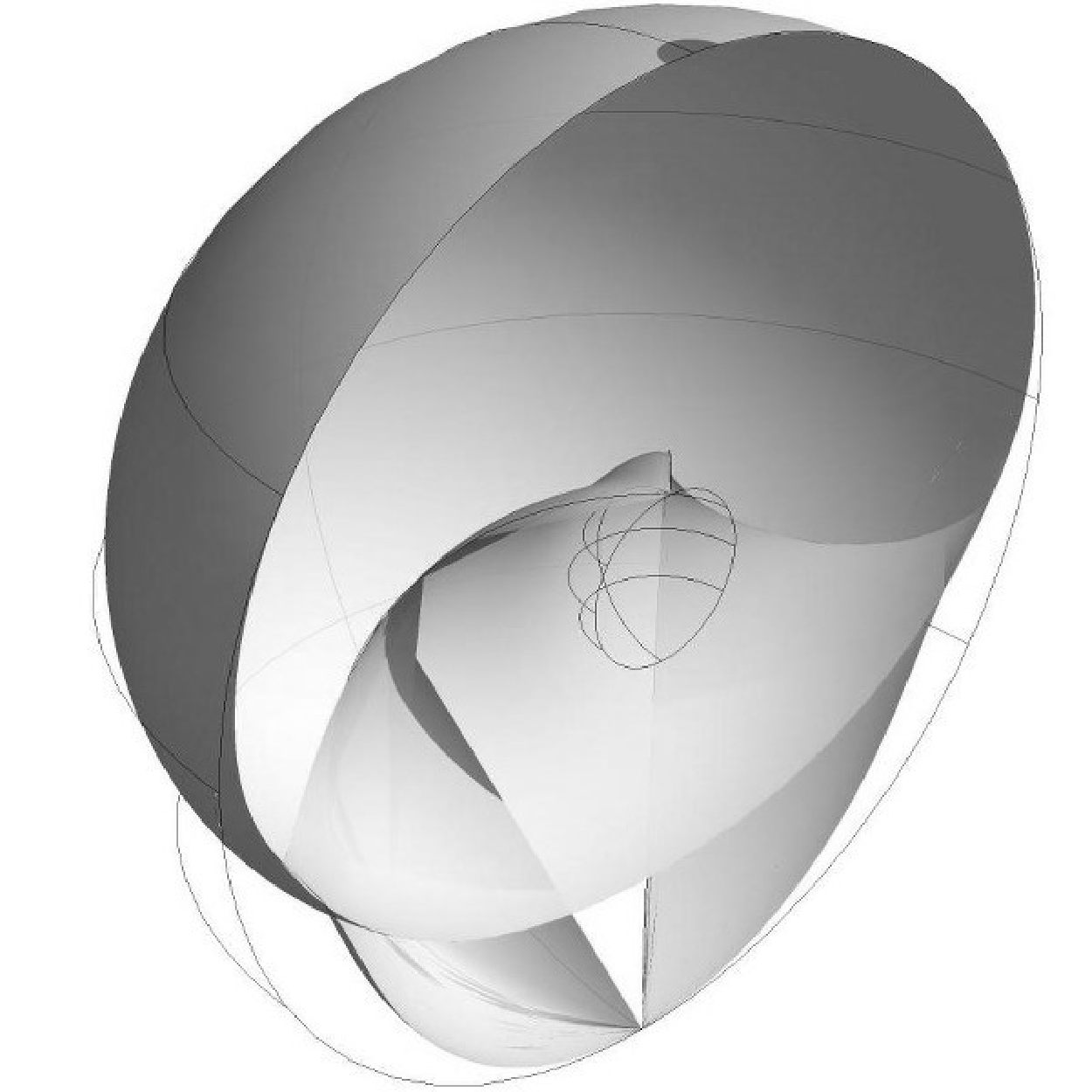} 
\ifx\undefined\distribute
\else
\hspace{0.5cm}
\fi
       \includegraphics[width=3.6cm]{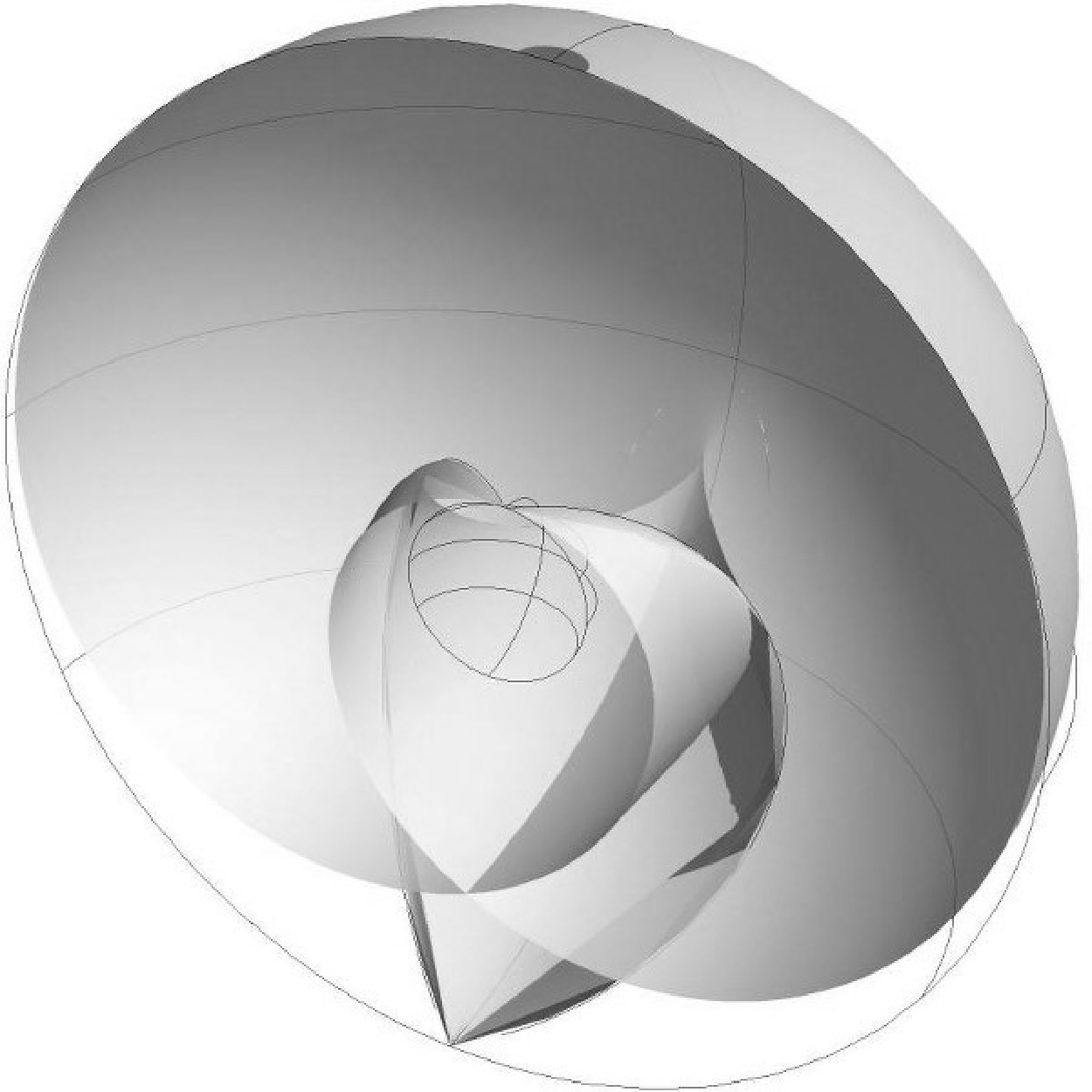} 
 \end{center}
 \caption{%
 The image of $f^{\I}_2$ (left) and halves of it (center and right). 
 }%
\label{fig:typeI-f2}
\end{figure}
\ifx\undefined\distribute
\begin{figure}[htb]%
\else
\begin{figure}
\fi
 \begin{center}
       \includegraphics[width=3.6cm]{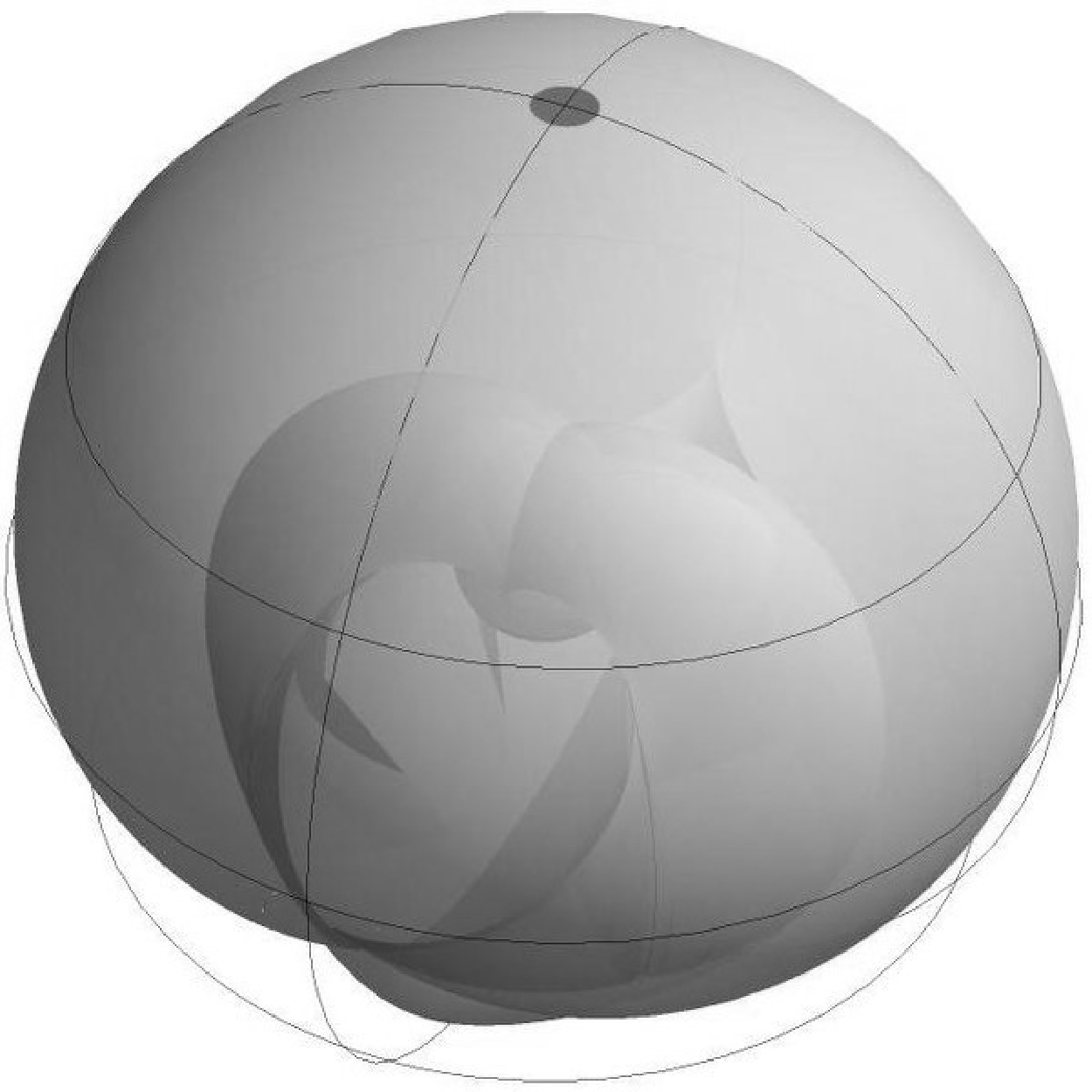} 
\ifx\undefined\distribute
\else
\hspace{0.5cm}
\fi
       \includegraphics[width=3.6cm]{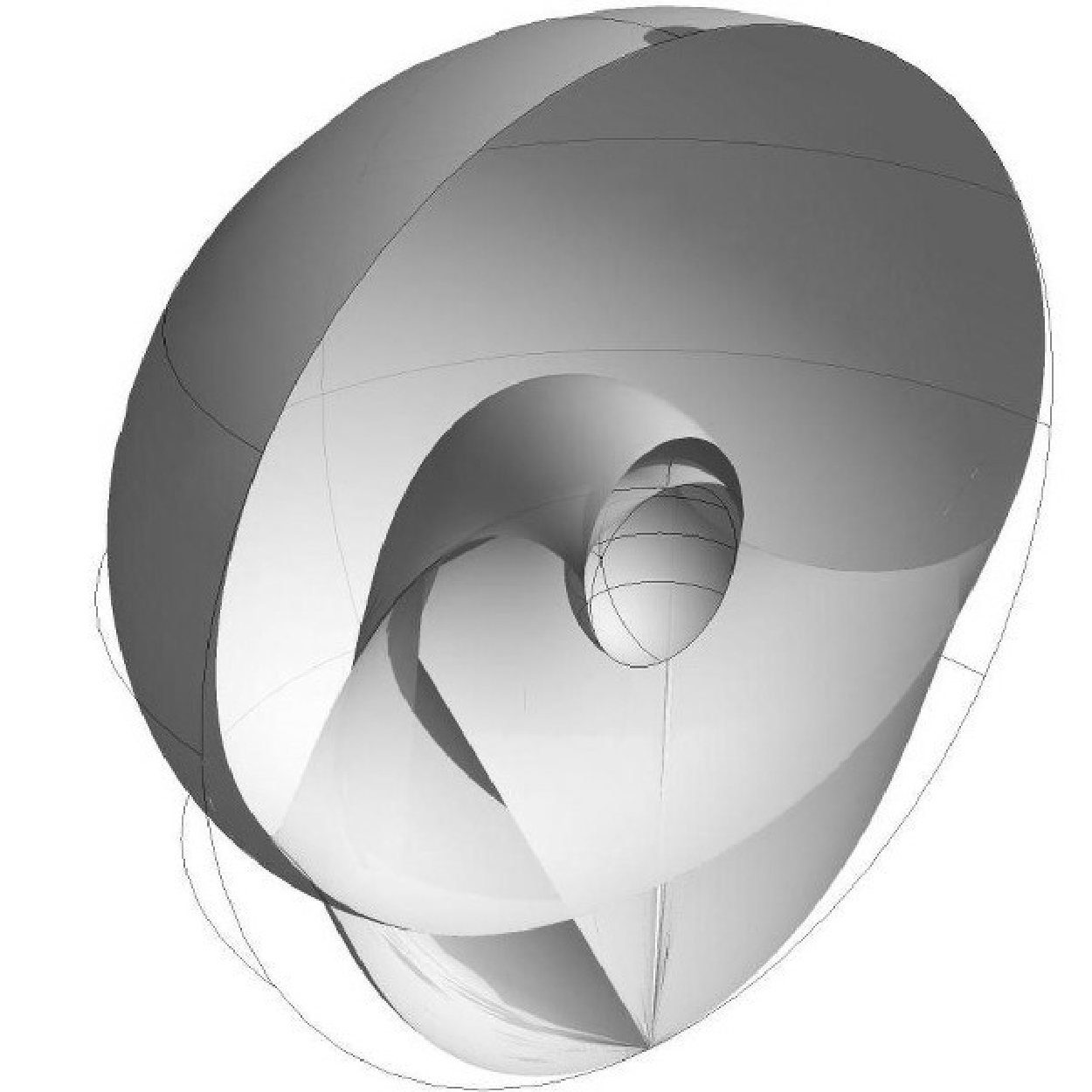} 
\ifx\undefined\distribute
\else
\hspace{0.5cm}
\fi
       \includegraphics[width=3.6cm]{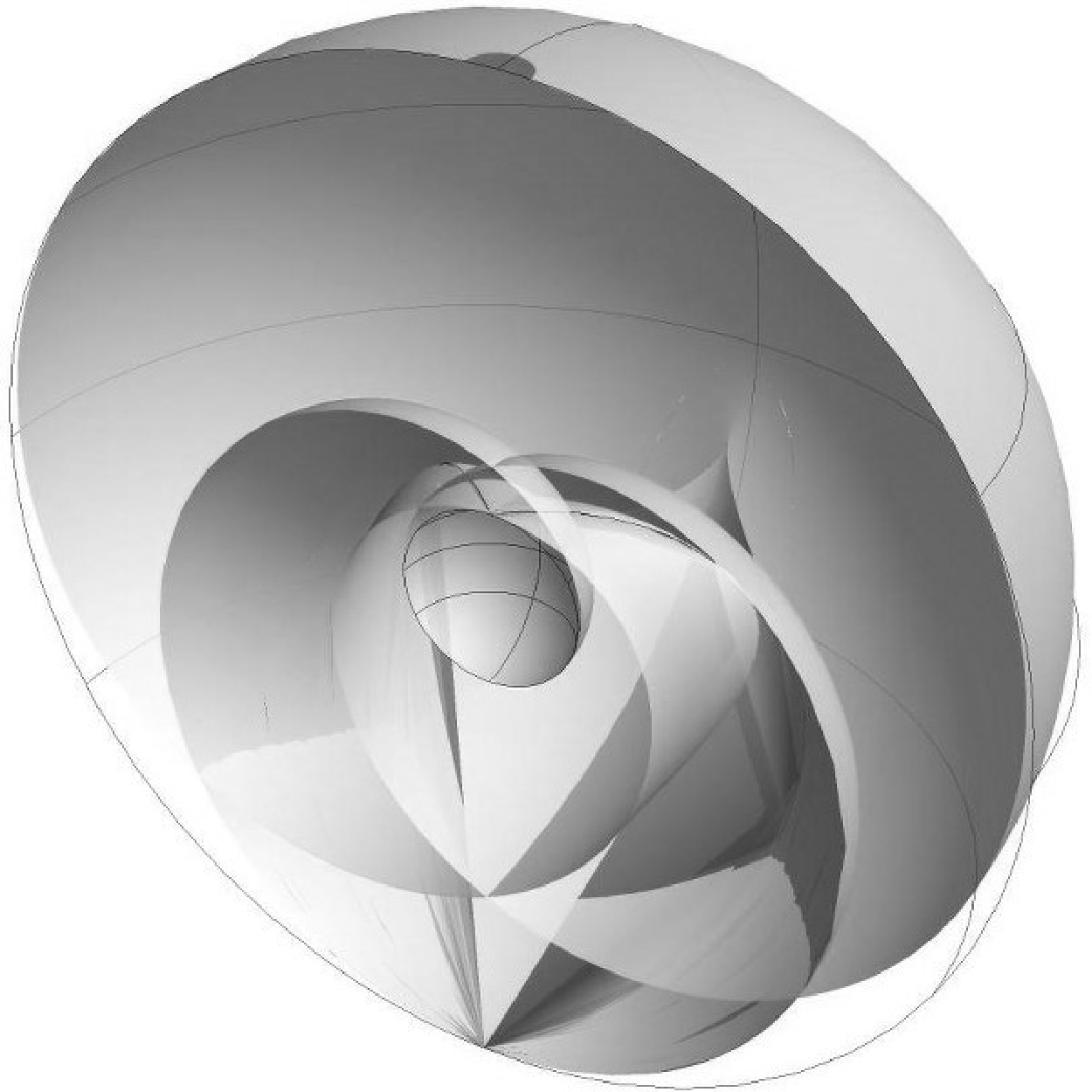} 
 \end{center}
 \caption{%
 The set $\CC^{\I}_2$ (left) and halves of it (center and right). 
 }%
\label{fig:typeI-C2}
\end{figure}
\ifx\undefined\distribute
\begin{figure}[htb]%
\else
\begin{figure}
\fi
 \begin{center}
       \includegraphics[width=3.6cm]{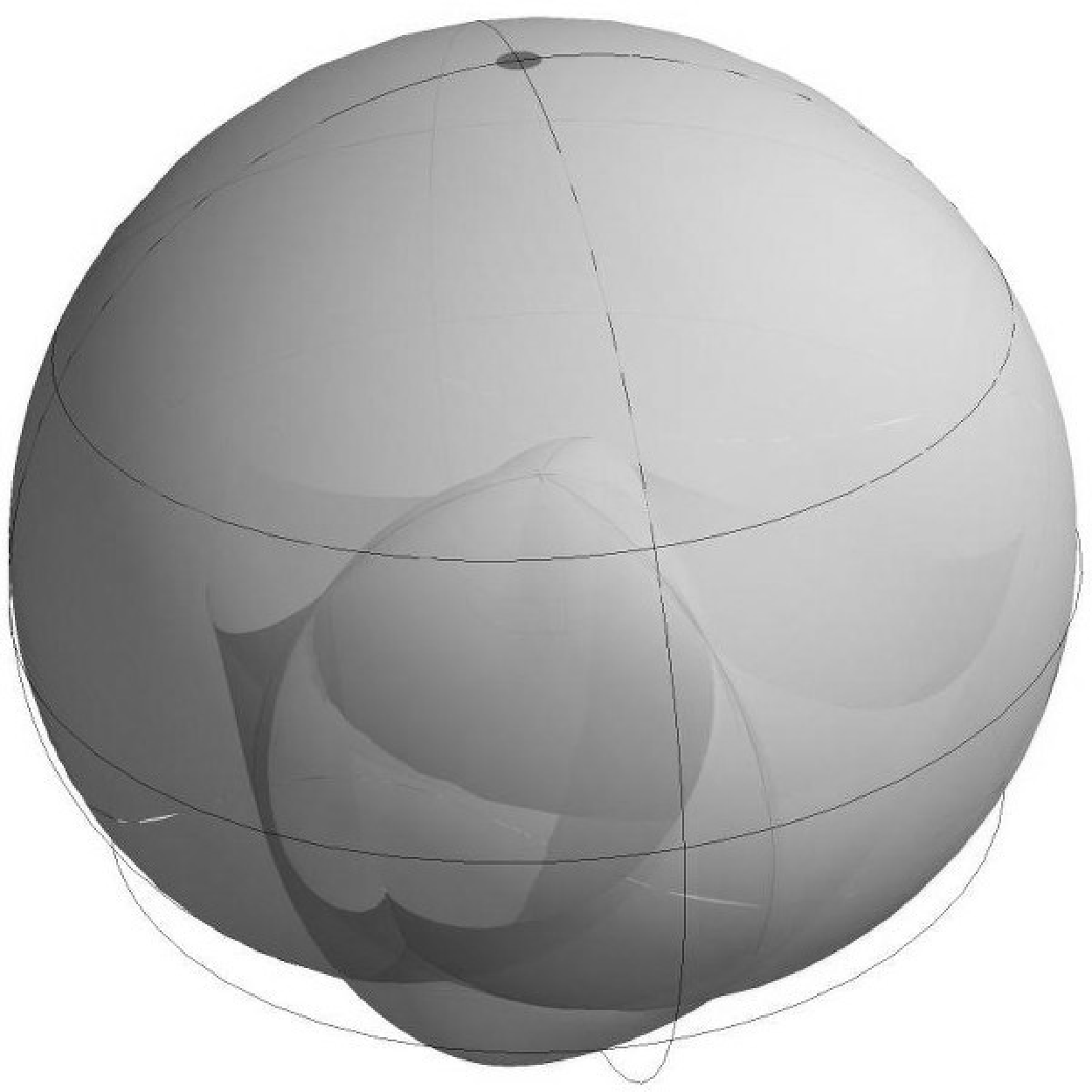} 
\ifx\undefined\distribute
\else
\hspace{0.5cm}
\fi
       \includegraphics[width=3.6cm]{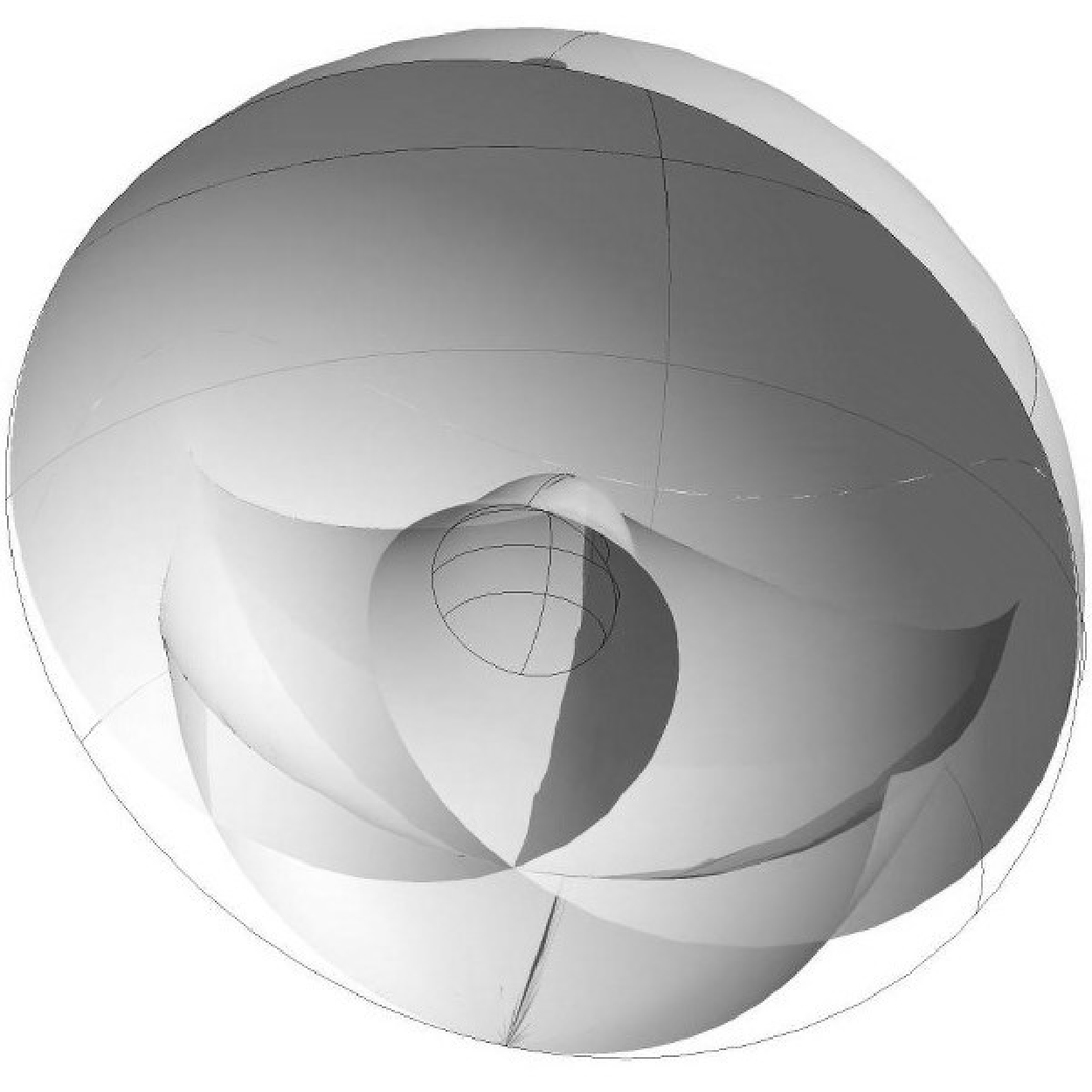} 
\ifx\undefined\distribute
\else
\hspace{0.5cm}
\fi
       \includegraphics[width=3.6cm]{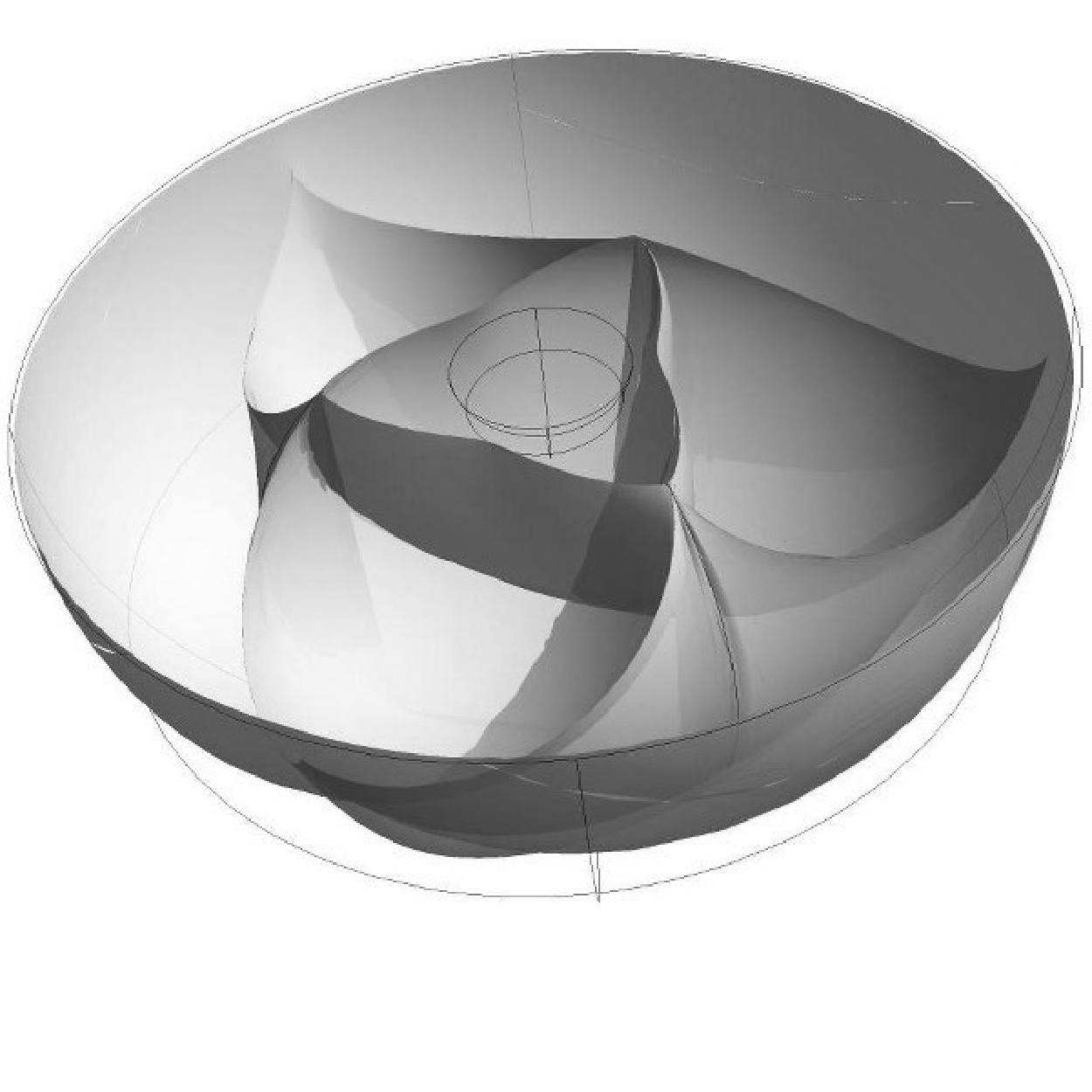} 
 \end{center}
 \caption{%
 The set $\CC^{\I}_3$ (left) and halves of it (center and right). 
 }%
\label{fig:typeI-C3}
\end{figure}
In Section~\ref{sec:type1},
we investigate the geometric properties of
$f^{\I}_m$, and show that the image of each $f^{\I}_m$
has an analytic extension whose image is
immersed outside of a compact set. 
See Figures\ \ref{fig:typeI-f2}, \ref{fig:typeI-C2}, \ref{fig:typeI-C3},
where
$f^{\I}_2$, $\CC^{\I}_2$ and
$\CC^{\I}_3$ are drawn in the stereographic 
hollow ball 
model \eqref{eq:hollow}. 
In Section~\ref{sec:type2}, we show that
each $\CC^{\II}_m$ can be realized as a warped product of
a certain trochoid and hyperbola.
In particular, $\CC^{\II}_2$ and $\CC^{\II}_3$
are almost embedded,
and $\CC^{\II}_{m}$ ($m\ge 4$) 
are almost immersed (cf.  Section 3). 
See Figures\ \ref{fig:h2}, \ref{fig:h3},
where
$f^{\II}_2$, $\CC^{\II}_2$ and
$\CC^{\II}_3$ are drawn in the stereographic 
hollow ball 
model \eqref{eq:hollow} as well.
\ifx\undefined\distribute
\begin{figure}[htb]%
\else
\begin{figure}
\fi
 \begin{center}
       \includegraphics[width=3.8cm]{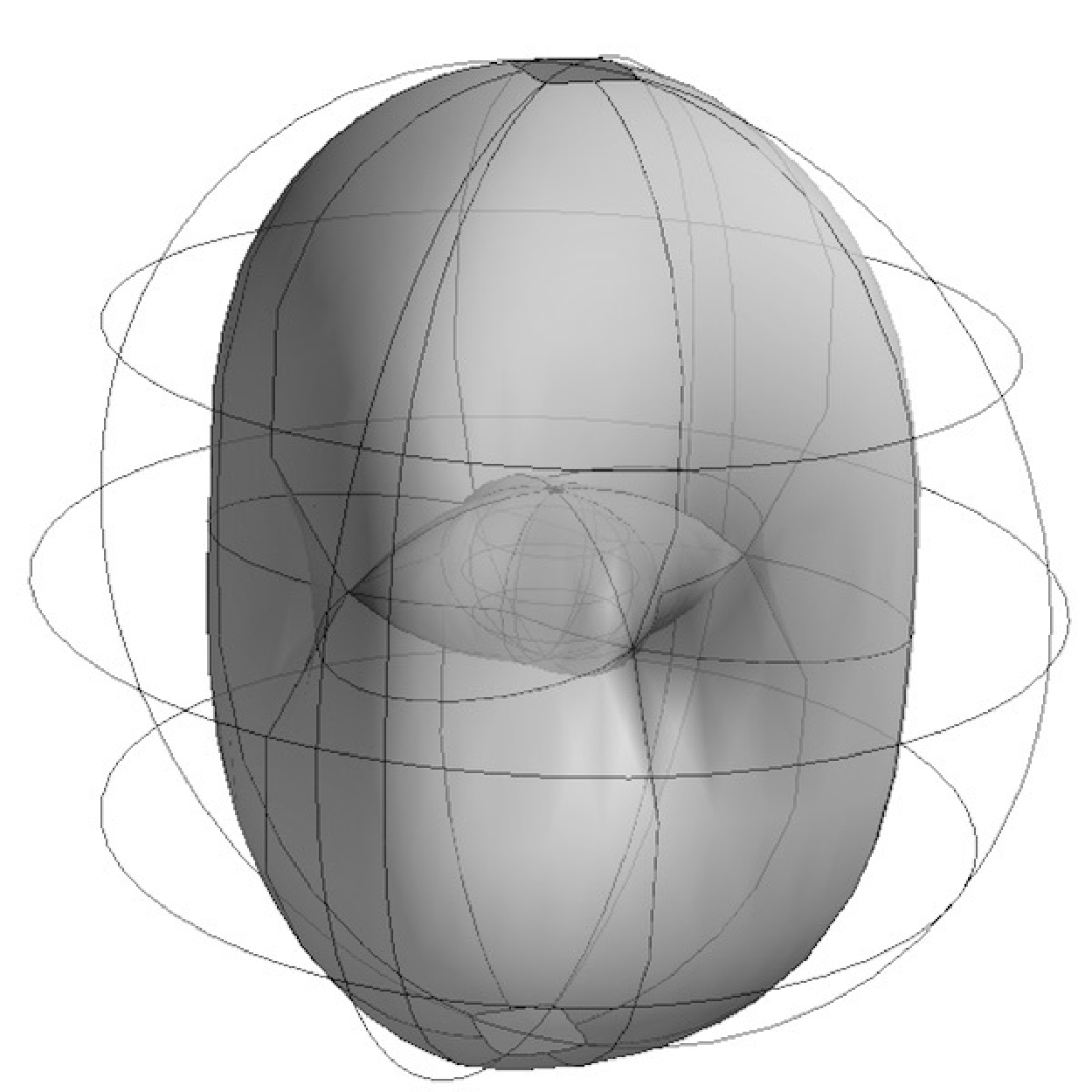} 
\ifx\undefined\distribute
\else
\hspace{1.5cm}
\fi
       \includegraphics[width=3.8cm]{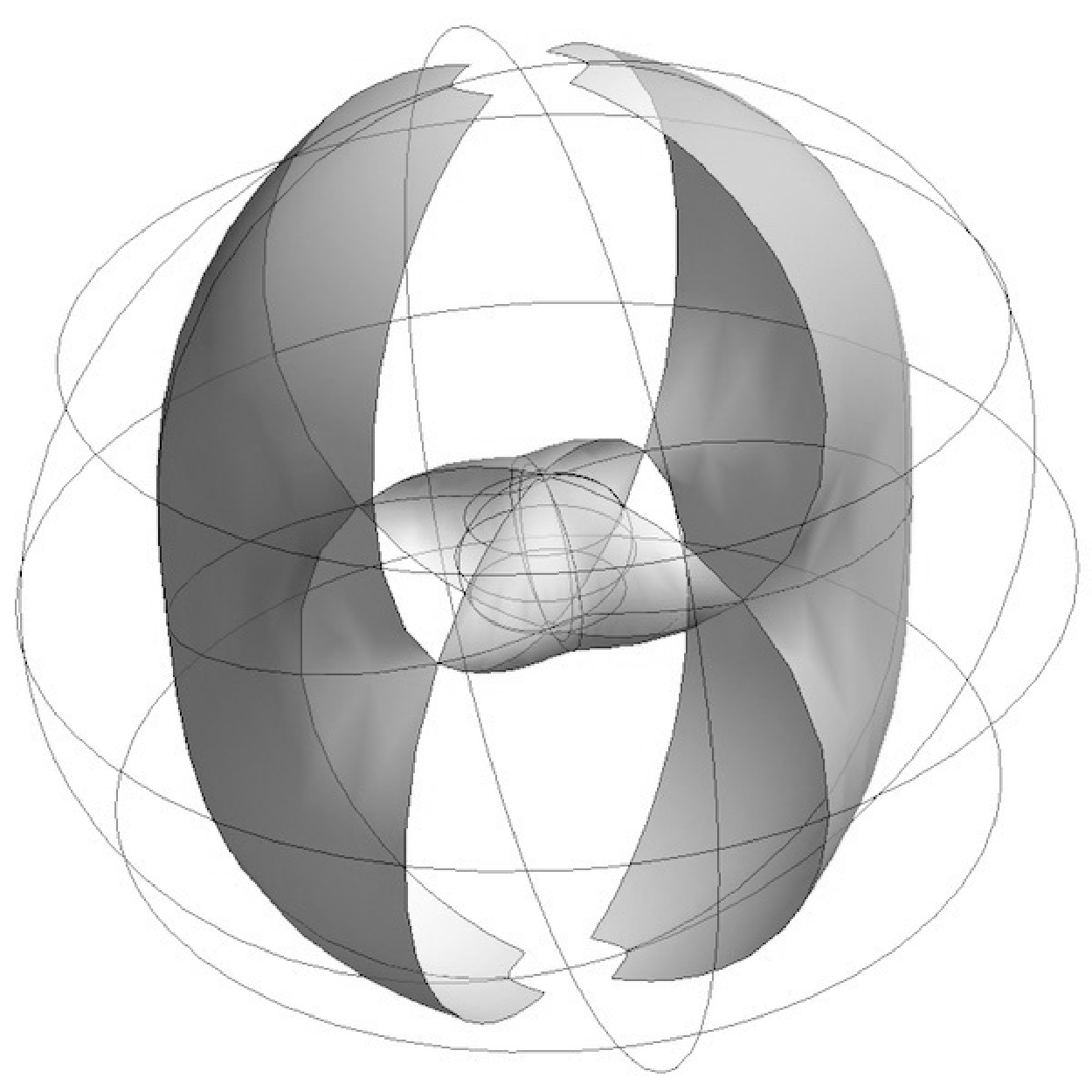} 
 \end{center}
 \caption{%
 The set $\CC^{\II}_2$ (left)
 and the image of $f^{\II}_2$ (right). 
 }%
\label{fig:h2}
\end{figure}
\ifx\undefined\distribute
\begin{figure}[thb]%
\else
\begin{figure}
\fi
 \begin{center}
       \includegraphics[width=3.8cm]{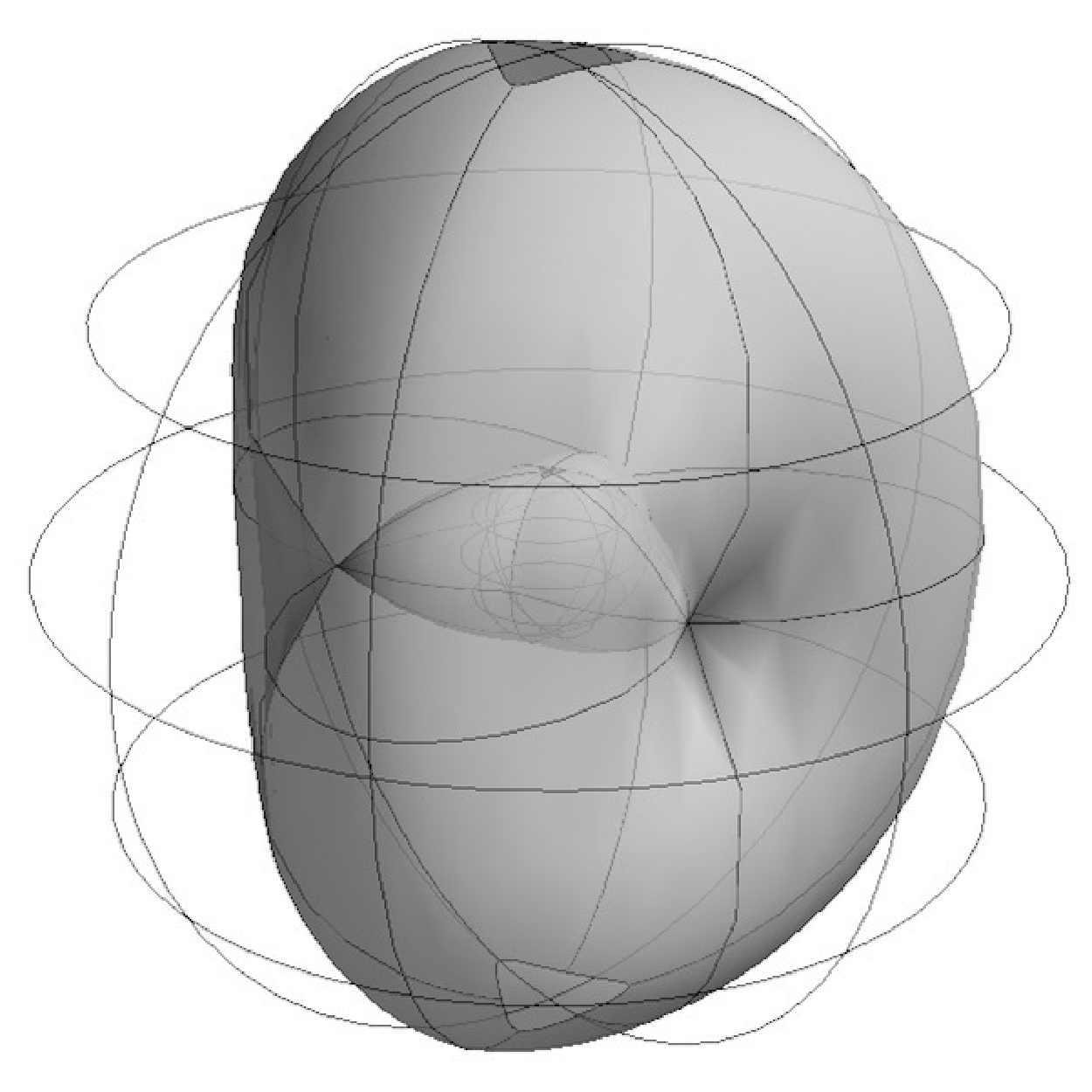} 
\ifx\undefined\distribute
\else
\hspace{1.5cm}
\fi
       \includegraphics[width=3.0cm]{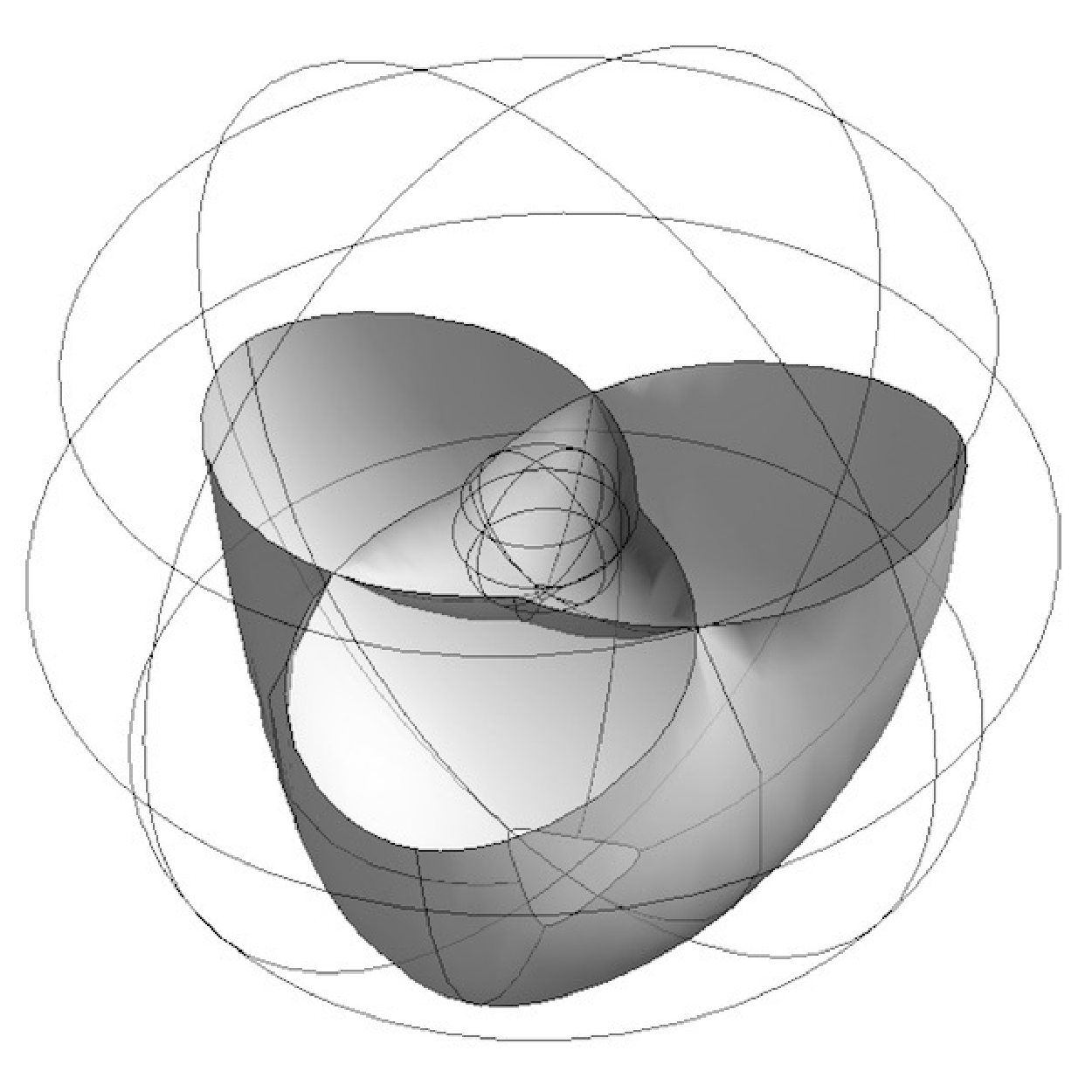} 
 \end{center}
 \caption{%
  The set $\CC^{\II}_3$ and half of it. 
 }%
\label{fig:h3}
\end{figure}

It is well-known that the de Sitter space $S^3_1$ 
can be compactified by including two
spheres $\partial_{\pm} S^3_{1}$.
These two sets $\partial_\pm S^3_{1}$
are called the \emph{ideal boundaries}.
In the stereographic hollow ball 
model, the relations 
\begin{equation}\label{eq:bdry}
   \partial_\pm S^3_{1}
      =\{\xi\in \R^3 \,;\, |\xi|=\sqrt{2}\mp 1\}
\end{equation}
hold (cf.~\eqref{eq:model}).
If a subset $\A$ of $S^3_1$ 
is closed, 
then each element of the set
\[
   \overline{\varPi(\A)} \cap \partial \D^3
       ~(\subset \partial_- S^3_{1}
         \cup \partial_+ S^3_{1})
\]
is called an \emph{endpoint},
where $\overline{\varPi(\A)}$ is the closure of 
$\varPi(\A)$ in $\R^3$.
Then the set 
$\overline{\varPi(\CC^{\J}_m)}\cap \partial_+ S^3_1$
consists of one (resp.\ two) point(s)
if $\J=\I$ and $m$ is odd
(resp.\ if $\J=\I$ and $m$ is even, or 
$\J=\II$).
On the other hand,
$\overline{\varPi(\CC^{\J}_m)}\cap \partial_- S^3_1$
always consists  of two points, that is,
the number of the endpoints of 
$\CC^{\J}_m$ ($\J=\I,\II$)
is three or four (cf.\ Theorems \ref{thm:A1} and \ref{thm:B}).
This is a remarkable phenomenon,
since other elliptic catenoids in $S^3_1$
do not have any analytic extensions
and have exactly two endpoints.

\section{Exceptional catenoids of type I.}
\label{sec:type1}

In this section, we show that
 $f^{\I}_m$
has an analytic extension.
For each integer $m\ge 2$,
we set 
\[
  f^{\I}_m(r,\theta)
   =\bigl(x_0(r,\theta),x_1(r,\theta),x_2(r,\theta),x_3(r,\theta)\bigr),
\]
with $z=r e^{\imag \theta}$ ($r>0$, $\theta\in S^1:=\R/2\pi\Z$).
Then
\begin{equation}\label{eq:1a}
\ifx\distribute\undefined
\mbox{\small
   $
\fi
    x_0 \pm x_3
    =\dfrac{m^2-1}{4m}r^{\pm 1}
    \left(2 \cos m\theta-\dfrac{m\mp 1}{m\pm 1}r^m\right),
\ifx\distribute\undefined
$,}
\fi
\end{equation}
\ifx\distribute\undefined
\begin{multline}
\label{eq:1b}
 x_1+\imag x_2
    =
 \mbox{\small$
    \dfrac{(m-1)^2}{4m} e^{\imag (m+1)\theta}
            +\dfrac{(m+1)^2}{4m} e^{-\imag (m-1)\theta}$}\\
   \mbox{\small$-e^{\imag\theta}\dfrac{m^2-1}{4m}r^m$}.
\end{multline}
\else
\begin{equation}
\label{eq:1b}
 x_1+\imag x_2
    =
    \dfrac{(m-1)^2}{4m} e^{\imag (m+1)\theta}
            +\dfrac{(m+1)^2}{4m} e^{-\imag (m-1)\theta}
   -e^{\imag\theta}\dfrac{m^2-1}{4m}r^m.
\end{equation}
\fi
We know that
$f^{\I}_m(r,\theta)$ has self-intersections, 
since it contains swallowtail singularities.
The limit curve
\begin{equation}\label{eq:star}
    \gamma_m(\theta):=\lim_{r\to 0}(x_1,x_2)
\end{equation}
gives a closed regular planar curve.

\ifx\undefined\distribute
\begin{figure}[htb]%
\else
\begin{figure}
\fi
 \begin{center}
\ifx\undefined\distribute
       \includegraphics[width=2.0cm]{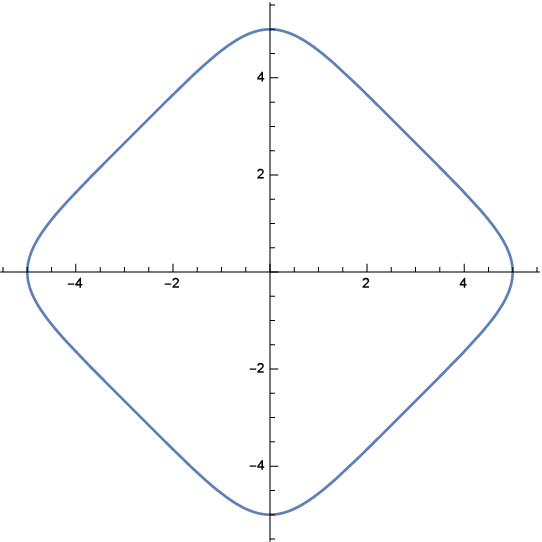} 
       \includegraphics[width=2.0cm]{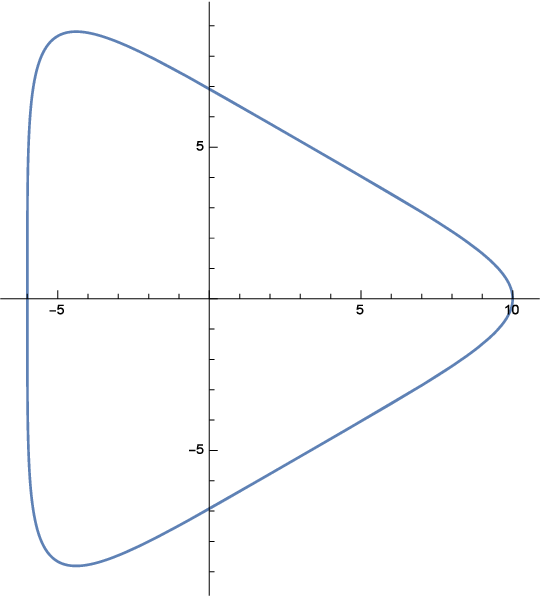} 
       \includegraphics[width=2.3cm]{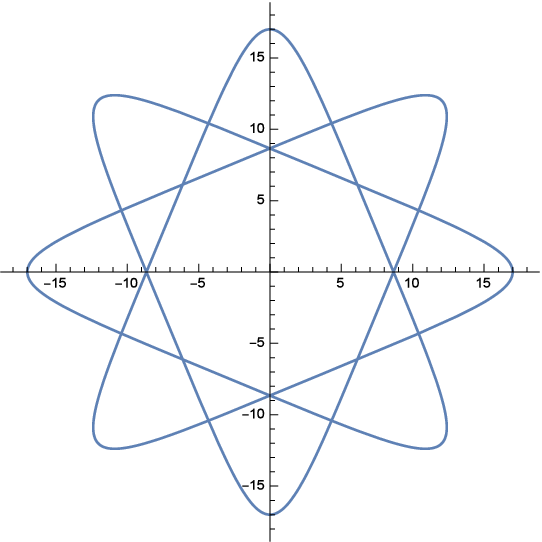} 
\else
       \includegraphics[width=0.2\textwidth]{tro2.eps} 
\hspace{1cm}
       \includegraphics[width=0.2\textwidth]{tro3.eps} 
\hspace{1cm}
       \includegraphics[width=0.23\textwidth]{tro4.eps} 
\fi
 \end{center}
 \caption{%
    The trochoids for $m=2,3,4$. 
 }%
\label{fig:tro}
\end{figure}

A \emph{hypo-trochoid} is a roulette traced by a point 
attached to a disk of radius $r_c$ rolling along 
the inside of a fixed circle of radius $r_m$, 
where the point is a distance $d$ from 
the center of the interior circle.
The parametrization of a hypo-trochoid is
given by
\ifx\distribute\undefined
\begin{align*}
     x(s) &= 
     \mbox{\small$  (r_c-r_m)\cos s + d \cos
               \left(\dfrac{r_c-r_m}{r_m}s\right)$},\\
     y(s) &= 
     \mbox{\small$(r_c-r_m)\sin s - d \sin 
                \left(\dfrac{r_c+r_m}{r_m}s\right)$}.
\end{align*}
\else
\begin{align*}
     x(s) &= 
    (r_c-r_m)\cos s + d \cos
               \left(\dfrac{r_c-r_m}{r_m}s\right),\\
     y(s) &= 
    (r_c-r_m)\sin s - d \sin 
                \left(\dfrac{r_c+r_m}{r_m}s\right).
\end{align*}
\fi
We prove the following:

\begin{prop}\label{prop:sing-end}
The plane curve $\gamma_m(\theta)$
has the following properties:
\begingroup
\renewcommand{\theenumi}{(\alph{enumi})}
\renewcommand{\labelenumi}{(\alph{enumi})}
\begin{enumerate}
\item\label{item:S0:1} 
      $\gamma_m(\theta+\pi)=(-1)^{m+1}\gamma_m(\theta)$
      for $\theta\in \R$,
\item\label{item:S0:2}
      the image of $\gamma_m$ is
      a convex curve if $m=2,3$,
\item\label{item:S0:3} $\gamma_m$ is a hypo-trochoid with
      (cf.\ Figure~\ref{fig:tro})
\ifx\distribute\undefined{\small
\fi
\[
r_c = \frac{m-1}{2},\quad r_m = \frac{m^2-1}{4m},
\quad d=\frac{(m+1)^2}{4m}.
\]
\ifx\distribute\undefined
}
\fi
\end{enumerate}
\endgroup
\end{prop}

\begin{proof}
The first two assertions
follow immediately.
The last assertion
follows from the expressions
\ifx\distribute\undefined
{\small
\fi
 \begin{align*}
    x_1 &= 
\frac{(m-1)^2\cos(m+1)\theta+(m+1)^2\cos(m-1)\theta}
{4m},\\
    x_2 &= 
\frac{(m-1)^2\sin(m+1)\theta-(m+1)^2\sin(m-1)\theta}{4m}.\qedhere
 \end{align*}
\ifx\distribute\undefined
}
\fi
\renewcommand{\qed}{\relax}
\end{proof}
We set
$\Omega:=\Omega^{+}\cup\Omega^{-}$,
where
\[
  \Omega^{\pm}:=\{(r,\theta)\in\R\times S^1\,;\,\pm r>0\}
   \quad\bigl(S^1:=\R/2\pi\Z\bigr).
\]
The expressions \eqref{eq:1a} and \eqref{eq:1b} are 
meaningful 
for $r<0$ as well, and $f_m^{\I}$ can be extended to $\Omega$.
We denote this extension by
$\tilde f_m^{\I}\colon{}\Omega \to S^3_1$.
If $m$ is odd, then
\begin{equation}\label{eq:mob-sym}
\tilde f_m^{\I}(-r,\theta+\pi)=\tilde f_m^{\I}(r,\theta).
\end{equation}
In particular, if $m$ is odd,
the image of $f_m^{\I}$ coincides
with that of $\tilde f_m^{\I}$.
On the other hand, if $m$ is even,  
\[
   \tilde f_m^{\I}(-r,\theta)=\iota\circ \tilde f^{\I}_m(r,\theta),
\]
where $\iota$ is the isometric involution given by
\begin{equation}\label{eq:involution}
 \iota\colon{}S^3_1\ni (t,x,y,z)\mapsto (-t,x,y,-z)\in S^3_1.
\end{equation}
Thus, if $m$ is even,
$f^{\I}_m(\Omega^{+})$
and  
$f^{\I}_m(\Omega^{-})$ 
are  congruent, but do not coincide with each other.
The singular set of $\tilde f_m^{\I}$ is
$\Sigma_m:=\Sigma_m^+\cup\Sigma_m^{-}$,
where 
\[
  \Sigma_m^{\pm}:=\{(r,\theta)\in\Omega^{\pm}\,;\,r^m+2 \cos m\theta=0\},
\]
each of which consists of $m$ components.
The image of each component of  the
singular set is a curve with singularities 
which is bounded in $S^3_1$, whose endpoints are
\begin{equation}\label{eq:Pk}
  P_k:=\bigl(0,\gamma_m(\alpha_k),0\bigr), 
  \qquad \alpha_k:=\frac{2k+1}{2m}\pi
  \qquad (k=0,\dots,2m-1).
\end{equation}
We denote by $A_m^{\pm}$ the domain in $\Omega^{\pm}$
containing a neighborhood of $r=\pm\infty$,
and $B_m^{\pm}:=\Omega^{\pm}\setminus \overline{A_m^{\pm}}$.
Then we have the expressions
\begin{align}
   A^{\pm}_m&=
   \{(r,\theta)\in\Omega^{\pm}\,;\,
                   \epsilon^m(r^m+2\cos m\theta)>0\},
   \label{eq:Am}\\
   B^{\pm}_m&=\{(r,\theta)\in\Omega^\pm\,;\,
                    \epsilon^m(r^m+2\cos m\theta)<0\},
   \label{eq:Bm}
\end{align}
where $\epsilon$ is the sign of $r$
(cf.\ Figure\ \ref{fig:domain}).
\begin{figure}
 \centering
 \begin{tabular}{c}
\ifx\distribute\undefined
 \includegraphics[width=0.35\textwidth]{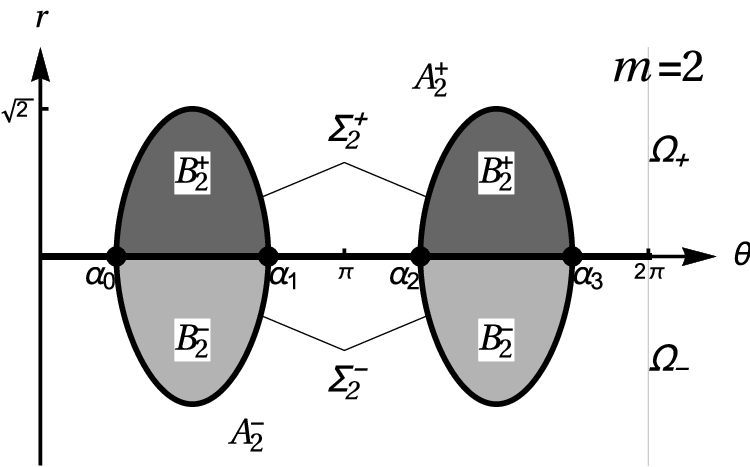}\\
 \includegraphics[width=0.35\textwidth]{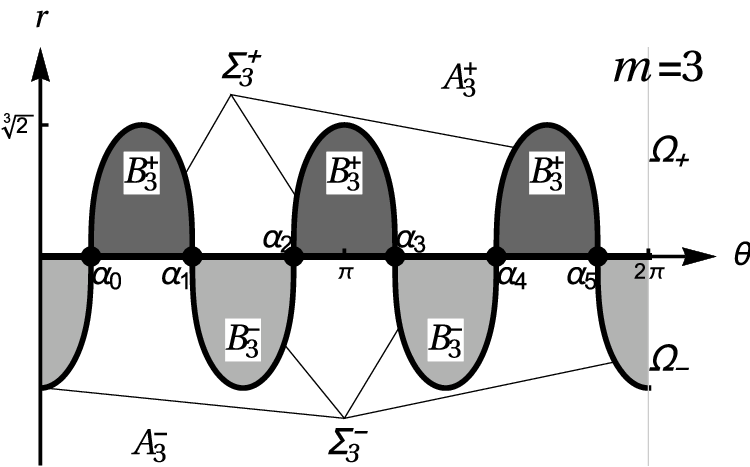}\\
\else
 \includegraphics[width=0.45\textwidth]{region-1.eps}\hspace{1cm}
 \includegraphics[width=0.45\textwidth]{region-2.eps}\\
\fi
 \end{tabular}
 \caption{The domains of $\tilde f_m^{\I}$ and their singular sets.}
 \label{fig:domain}
\end{figure}
We next consider the light-like lines
\[
  L_k:=
    \{(t,\gamma_m(\alpha_k),-t)\,;\, t\in \R\}\subset S^3_1
\]
passing through $P_k$ 
for $k=0,1,\dots,2m-1$,
and set
\[
  \CC_m^{\I}:= \tilde f_m^{\I}(\Omega)\cup L_0\cup\dots \cup L_{2m-1}.
\]
Then $\CC_m^{\I}$ is the analytic extension of $f_m^{\I}$.
In fact,
\begin{thm}\label{thm:A}
For each integer $m\geq 2$,
\begingroup
\renewcommand{\theenumi}{(\roman{enumi})}
\renewcommand{\labelenumi}{(\roman{enumi})}
 \begin{enumerate}
  \item\label{item:A:1} 
        $\CC^{\I}_m$ is a closed set of $S^3_1$.
        In particular, 
        if $m$ is odd, then
	$\CC^{\I}_m$ is the closure of the
	image of  $f^{\I}_m$. 
	On the other hand, if $m$ is even,
	then the closure of the image of $f^{\I}_m$
	is just half of $\CC^{\I}_m$.
	The other half can be obtained by 
	the isometric involution $\iota$ of $S^3_1$
	given in \eqref{eq:involution}.
  \item\label{item:A:2} 
 Moreover, $\CC^{\I}_m$ is analytically immersed outside 
        the compact set consisting of 
        the image of $\Sigma_m$,
        and the points $\{(0,\gamma_m(\alpha_k),0)\,;
\,k=0,\dots,2m-1\}$.
 \end{enumerate}
\endgroup
\end{thm}
\begin{proof}
By \eqref{eq:1a}, $x_0(r,\theta)$ diverges for $r\to \pm\infty$.
Take a sequence $\{\zeta_j=(r_j,\theta_j)\}_{j=1,2,\dots}$
on $\Omega$ such that $\displaystyle\lim_{j\to\infty}r_j=0$.
Taking a subsequence if necessary, we may assume $\{\zeta_j\}$
is included in $\Omega^+$ or $\Omega^-$, and 
$\displaystyle\lim_{j\to\infty}\theta_j=\beta$.
If $\cos m\beta\neq 0$, 
\eqref{eq:1a} implies that $\displaystyle\lim_{j\to\infty} x_0(r_j,\theta_j)$
diverges.
On the other hand, if $\cos m\beta=0$, that is, $\beta=\alpha_k$
for some $k$,  then 
$\displaystyle\lim_{j\to\infty}(x_0(\zeta_j)+x_3(\zeta_j))$ tends to
$0$, that is,
$\tilde f^{\I}_m(\zeta_j)$ is asymptotic to the line $L_k$.
Conversely, for each point $Q_{k,t}:=(t,\gamma_m(\alpha_k),-t)\in L_k$,
we set 
\[
   \zeta_j := \left(\frac{1}{j}, \frac{1}{m}\cos^{-1}\frac{4mt}{j(m^2-1)}
             \right)\quad(j=1,2,\dots),
\]
where $\cos^{-1}$ is considered as a map
\begin{equation}\label{eq:cos-inv}
   \cos^{-1}\colon{}(-1,1)\to
    \left(m\alpha_k-\frac{\pi}{2},m\alpha_k+\frac{\pi}{2}\right).
\end{equation}
Then $\displaystyle\lim_{j\to\infty}\zeta_j=(0,\alpha_k)$
and $\displaystyle\lim_{j\to\infty}\tilde f^{\I}_m(\zeta_j)=Q_{k,t}$, 
where $\alpha_k$ is as in \eqref{eq:Pk}.
Hence $\CC^{\I}_m$ is the closure of the image of $\tilde f_m^{\I}$,
proving the first part of \ref{item:A:1}.
The second part of \ref{item:A:1} is already proven.
We next prove (ii). Since $\tilde f^{\I}_m$ is an analytic 
immersion on $\Omega\setminus\Sigma_m$,
it is sufficient to show that 
$\CC^{\I}_m$ is parametrized analytically on a
neighborhood of $L_k$, which gives an immersion
on $L_k\setminus \{P_k\}$.
For this purpose, we set
$s:=(\cos m\theta)/r$.
Then the $x_j$ ($j=0,1,2,3$) have the following expressions:
\ifx\undefined\distribute
\begin{align*}
 x_0\pm x_3 &= 
     \text{\small$
            \dfrac{m^2-1}{4m}r^{\pm 1}
            \left(2rs-\dfrac{m\mp 1}{m\pm 1}r^{m}\right),
$}
\\
 x_1+\imag x_2&=
     \text{\small$
\dfrac{e^{\imag \cos^{-1}(sr)/m}}{4m}
           \biggl(
                (m-1)^2e^{\imag \cos^{-1}(sr)}
$}
\\
     &\hphantom{===}
     \text{\small$
+(m+1)^2 e^{-\imag \cos^{-1}(sr)}
                   -(m^2-1)r^m \biggr).
$}
\end{align*}
\else
\begin{align*}
 x_0\pm x_3 &= 
            \dfrac{m^2-1}{4m}r^{\pm 1}
            \left(2rs-\dfrac{m\mp 1}{m\pm 1}r^{m}\right),
\\
 x_1+\imag x_2&=
            \dfrac{e^{\imag \cos^{-1}(sr)/m}}{4m}
           \biggl(
                (m-1)^2e^{\imag \cos^{-1}(sr)}
                 +(m+1)^2 e^{-\imag \cos^{-1}(sr)}
                 (m^2-1)r^m \biggr).
\end{align*}
\fi
Since 
$\left.\bigl(\partial(x_1+\imag x_2)/\partial r\bigr)\right|_{(0,s)}\neq 0$
if $s\neq 0$, 
one can easily check that $\tilde f^{\I}_m(r,s)$ 
is an immersion at $(0,s)$ for each $s\in \R\setminus \{0\}$,
which proves the assertion.
\end{proof}
Next, we consider the endpoints of $\CC_m^{\I}$.
Let
\begin{equation}\label{eq:E}
\begin{aligned}
  p_{\pm}&:=\bigl(0,0,\pm(\sqrt{2}-1)\bigr)\in \partial_{+}S^3_1,\\
  n_{\pm}&:=\bigl(0,0,\pm(\sqrt{2}+1)\bigr)\in \partial_{-}S^3_1,
\end{aligned}
\end{equation}
where $\partial_{\pm} S^3_1$ are
the ideal boundaries 
given in \eqref{eq:bdry}.
We set
\begin{equation}\label{eq:y0}
   y:=(y_1,y_2,y_3):=\varPi \circ f^{\I}_m
         = \frac{1}{\delta}(x_1,x_2,x_3),
\end{equation}
where $\delta=x_0+\sqrt{2x_0^2+1}$ (cf.\ \eqref{eq:hollow}).
\begin{thm}\label{thm:A1}
   If $m$ is even {\rm(}resp.\ odd{\rm)},
   the set of endpoints of $\CC^{\I}_m$ is
   $\{p_{\pm},n_{\pm}\}$ 
   {\rm(}resp.\ $\{p_{-},n_{\pm}\}${\rm)}.
   More precisely, let $\{\zeta_j=(r_j,\theta_j)\}$
   be a sequence in $\Omega$ whose image under $\tilde f^{\I}_m$
   is unbounded.
   Then the following cases occur{\rm:}
   \begin{enumerate}
    \item\label{item:A1:1} 
$\displaystyle\lim_{j\to\infty}y(\zeta_j)=n_-$ holds
when $\displaystyle\lim_{j\to\infty}r_j=+\infty$
$($that is, $\{\zeta_j\}$ lies in $\Omega^+$ and
diverges$)$.
    \item\label{item:A1:2}  
	  When $\displaystyle\lim_{j\to\infty}r_j=-\infty$,
that is, if $\{\zeta_j\}$ lies in $\Omega^-$ and
diverges, then 	  $\displaystyle\lim_{j\to\infty}y(\zeta_j)$ is
	  $p_+$ {\rm(}resp.\ $n_-${\rm)}
	  if $m$ is even {\rm(}resp.\ odd{\rm)}.
    \item\label{item:A1:3} 
	  When $r_j\to 0$ and $\{\zeta_j\}$ is contained in 
	  $A_m^+$ {\rm(}resp.\ $A_m^-$, $B_m^+$, $B_m^-${\rm)},
	  the limit of $y(\zeta_j)$ is obtained as in the following 
	  table{\rm:}
	  \[
	    \begin{array}{|c||c|c|c|c|}
	     \hline
	     \text{The domain containing $\{\zeta_j\}$}
                & A_m^+ & A_m^- & B_m^+ & B_m^- \\
	     \hline
	      \displaystyle\lim_{j\to\infty}y(\zeta_j) \text{ for even $m$}
	        & p_- & n_+ & n_+ & p_-\\
             \hline
	      \displaystyle\lim_{j\to\infty}y(\zeta_j) \text{ for odd $m$}
	        & p_- & p_- & n_+ & n_+\\
	     \hline
	    \end{array}
	  \]
   \end{enumerate}
\end{thm}
\begin{proof}
 We rewrite \eqref{eq:y0}  as 
 \begin{equation}\label{eq:y}	
   y_l      = \frac{x_l/x_0}{1+\sgn(x_0)\sqrt{2+1/(x_0)^2}}
    \quad (l=1,2,3),
 \end{equation}
 where $\sgn(x_0)$ denotes the sign of $x_0$.
 By \eqref{eq:1a} and \eqref{eq:1b},
 \begin{align*}
 &\lim_{r\to \pm\infty} \frac{x_3}{x_0}=1,\quad\qquad
 \lim_{r\to \pm\infty} \frac{x_l}{x_0}=0 \quad (l=1,2),\\
 &\lim_{r\to +\infty} x_0= -\infty,\qquad
 \lim_{r\to -\infty} (-1)^mx_0= \infty,
 \end{align*}
 proving \eqref{item:A1:1} and \eqref{item:A1:2}.

We prove \eqref{item:A1:3} for the case that $\{\zeta_j\}\subset B_m^-$.
 Noticing that $r_j<0$, \eqref{eq:Bm} implies that
 \[
    (-1)^m \left(r_j^{m-1} +\frac{\cos m\theta_j}{r_j}\right)>0
 \]
holds for each $j$.
Since $\{x_0(\zeta_j)\}$ is unbounded,
 so is $(\cos m\theta_j)/r_j$.
 Then the sign
of $\cos m\theta_j/r_j$
is equal to
$(-1)^m$ for sufficiently large $j$
 because $r_j$ tends to $0$.
 Then by \eqref{eq:1a}, $\sgn\bigl(x_0(\zeta_j)\bigr)=(-1)^m$.
 On the other hand, \eqref{eq:1a} implies that
 $\displaystyle\lim_{j\to\infty}x_3(\zeta_j)/x_0(\zeta_j)=-1$.
 Thus, we have
 \[
   \lim_{j\to\infty}y_3(\zeta_j)
    = \frac{-1}{1+(-1)^m\sqrt{2}}
    = 1-(-1)^{m}\sqrt{2}.
 \]
 Since $x_1$ and $x_2$ are bounded near $r=0$, 
 $y_l(\zeta_j)$ tends to $0$ for $l=1$, $2$.
 Thus we have the conclusion.
 The other cases can be proved similarly.
 \end{proof}

\section{Exceptional catenoids of type II.}
\label{sec:type2}
Here we show that the image of the
exceptional catenoid $f^{\II}_m$
in $S^3_1$ has an analytic extension.
For each integer $m\ge 2$,
we set 
\[
  f^{\II}_m(r,\theta)
    =\bigl(
          x_0(r,\theta),
	  x_1(r,\theta),
	  x_2(r,\theta),
	  x_3(r,\theta)
	  \bigr),
\]
with
$z=re^{\imag \theta}$ ($r>0\,\,,\theta\in [0,2\pi)$).
By \eqref{eq:ec},
$f^{\II}_m$'s components are
\begin{equation}\label{eq:parametrization}
 \begin{aligned}
  x_0 &= 
\ifx\undefined\distribute
    \text{\small$
\fi
  \dfrac{1-m^2}{4m}\left(r+\frac{1}{r}\right)\cos m\theta
\ifx\undefined\distribute
$}
\fi,\\
  x_3 &= 
\ifx\undefined\distribute
    \text{\small$
\fi
  \dfrac{1-m^2}{4m}\left(r-\frac{1}{r}\right)\cos m\theta
\ifx\undefined\distribute
$}
\fi,\\
  x_1 &= 
%
\ifx\undefined\distribute
    \text{\small$
\fi
   -\frac{
                       (m^2+1)\cos m\theta\cos\theta +
                          2m  \sin m\theta\sin\theta}{2m}
\ifx\undefined\distribute
$}
\fi,\\
%
  x_2 &= 
\ifx\undefined\distribute
    \text{\small$
\fi
  -\frac{
                       (m^2+1)\cos m\theta\sin\theta -
                          2m  \sin m\theta\cos\theta}{2m}
\ifx\undefined\distribute
$}
\fi,
 \end{aligned}
\end{equation}
where $z=re^{\imag \theta}$ ($r>0,\,\,\theta\in [0,2\pi)$).
The secondary Gauss map $g_m$
of $f^{\II}_m$ is
a meromorphic function on $\C\cup\{\infty\}$
given by (cf.\ \cite[(39)]{FKKRUY2})
\[
   g_m = (z^m-1)/(z^m+1).
\]
Since the singular set $\Sigma_m$ of the map 
$f^{\II}_m$ is
\[
 \Sigma_m=\{z\in \C\setminus\{0\}\,;\, |g_m(z)|=1\}
 =\{r e^{\imag\theta}\in \C\setminus\{0\}\,;\, \cos m\theta=0\},
\]
we have 
 $\Sigma_m = \sigma_0\cup \sigma_1\cup\dots\cup\sigma_{2m-1}$,
where
\begin{equation}\label{eq:sigma-j}
   \sigma_k:=\left\{z=re^{\imag \alpha_k}\,;\,   r>0
        \right\}\quad
	\left(\alpha_k := \tfrac{(2k+1)\pi}{2m}\right)
\end{equation}
for $k=0,\dots,2m-1$.
In particular, if we set
\begin{equation}\label{eq:regular-set}
  \Omega_k:=\left\{re^{\imag \theta}\,;\,
	      \tfrac{(2k-1)\pi}{2m}
	      < \theta< \tfrac{(2k+1)\pi}{2m},
	      r>0 
        \right\},
\end{equation}
then the union of the $\Omega_k$ ($k=0,\dots,2m-1$)
is the regular set of $f^{\II}_m$, that is,
the regular set consists of a disjoint union of $2m$ sectors.

\begin{prop}\label{prop:sing-end-2}
 The map $f^{\II}_m$ satisfies{\rm:}
\begingroup
\renewcommand{\theenumi}{(\roman{enumi})}
\renewcommand{\labelenumi}{(\roman{enumi})}
 \begin{enumerate}
  \item\label{item:S:1}
       For each $m\ge 2$,
       the image $f^{\II}_m(\sigma_k)$ 
       consists of a point.
       More precisely,
       \[
	  f^{\II}_m(\sigma_k)=
	         (-1)^{k}\left(0,-\sin \alpha_k,
                  \cos \alpha_k,
                      0\right),
       \]
       where $\alpha_k$ is as in \eqref{eq:sigma-j}
       $(k=0,\dots,2m-1)$.
  \item\label{item:S:2} 
The endpoints of the image of $f^{\II}_m$ are 
        the four points 
        $p_{\pm}$ and $n_{\pm}$  as in 
        \eqref{eq:E}.
 \end{enumerate}
\endgroup
\end{prop}
 
\begin{proof}
 Substituting $\theta=\alpha_k$ into
 \eqref{eq:parametrization}
 and using  that
 $\cos m\theta=0$ and $\sin m\theta=(-1)^k$
 on $\sigma_k$, we get the first assertion.

 To prove the second assertion,
 we remark that
       \begin{equation}\label{eq:signx0}
          \sgn(x_0) = (-1)^{k+1}\qquad 
	   \text{(on $\Omega_k$)},
       \end{equation}
 for each $k$,
 since $\sgn(\cos m\theta)=(-1)^{k}$.
 Take a sequence $\{z_j\}$ on $\C\setminus\{0\}$
 such that $\varPi\circ f^{\II}_m(z_j)$ converges to
 one of the points in the ideal boundary.
 By \ref{item:S:1}, we may assume that each $z_j\not\in\Sigma_m$.
 With finitely many sectors,
 we may also assume $\{z_j\}\subset \Omega_k$
 for some $k$.
 Then $x_0(z_j)$ diverges to $\infty$ or $-\infty$ 
 as $j\to\infty$,
 that is,
 $\{r_j+r_j^{-1}\}_{j=1,2,\dots}$ is unbounded, where $r_j:=|z_j|$.
 Taking a subsequence, we may assume 
 \begin{equation}\label{eq:r-lim}
   \lim_{j\to\infty} r_j= 0 \quad \mbox{or}\quad
   \lim_{j\to\infty} r_j= \infty.
 \end{equation}
 We set $y:=\varPi\circ f^{\II}_m$.
 Since $x_1$ and $x_2$ are bounded 
 (cf.\ \eqref{eq:parametrization}), 
 $y_l(z_j)\to 0$ for $l=1,2$, where $y=(y_1,y_2,y_3)$.
 On the other hand, by \eqref{eq:parametrization} and \eqref{eq:r-lim},
 we have
 $\displaystyle\lim_{j\to\infty}\bigl(x_3(z_j)/x_0(z_j)\bigr)=\pm 1$.
 Thus, we have
 $\displaystyle\lim_{j\to\infty} y_3(z_j) = \pm\sqrt{2}\pm 1$,
 which proves  \ref{item:S:2}.
\end{proof}

It should be remarked that
$x_1$, $x_2$ depend only on
the variable $\theta$, and 
\ifx\undefined\distribute
$
\else
\[
\fi
\bigl(x_1(\theta),x_2(\theta)\bigr)=-2\gamma_m(\theta)
\ifx\undefined\distribute
$
\else
\]
\fi
holds.
Here, $\gamma_m$ is exactly the same
hypo-trochoid as given 
in Proposition~\ref{prop:sing-end}.
For fixed $\theta$, the image of the curve
defined by
$r \mapsto \bigl(x_0(r,\theta),x_3(r,\theta)\bigr)$
coincides with
\ifx\distribute\undefined
\begin{multline}\label{eq:hyp0}
 \biggl\{(t,z)\in \R^2_1\,;\,
   t^2-z^2=\frac{(m^2-1)^2}{(2m)^2}\cos^2 m\theta,
 \\ 
 \sgn(\cos m \theta)t<0
 \biggr\}.
\end{multline}
\else
\begin{equation}\label{eq:hyp0}
 \biggl\{(t,z)\in \R^2_1\,;\,
   t^2-z^2=\frac{(m^2-1)^2}{(2m)^2}\cos^2 m\theta,
 \sgn(\cos m \theta)t<0
 \biggr\}.
\end{equation}
\fi
In particular, it is half of
a hyperbola when $\cos m\theta\neq 0$.
If $\cos m\theta=0$, the image 
reduces to a point.
So we can conclude that
the real analytic extension of the image of
$f^{\II}_m$ coincides with the set
\begin{multline*}
 \CC^{\II}_m:=
 \biggl\{
     (t,x,y,z)\in \R^4_1\,;\,
     (x,y)=-2\gamma_m(\theta),\\
     t^2-z^2
 =\frac{(m^2-1)^2}{(2m)^2}\cos^2 m\theta,
 \,\,     \theta\in [0,2\pi) 
\biggr\}.
\end{multline*}
For each $k=0,\dots,2m-1$, the analytic extension
$\CC^{\II}_m$ 
contains a union of two light-like lines
\[
   L^{\pm}_k:=\{(t,-\sin \alpha_k,\cos
    \alpha_k,\pm t)\,;\, t \in \R\},
\]
where $\alpha_k$ is as in \eqref{eq:sigma-j}.
Moreover, $\CC^{\II}_m$ is symmetric with respect to the isometric
involution
\ifx\distribute\undefined
\begin{multline*}
 S^3_1 \ni
 (t,x,y,z)\mapsto
 (t, 
 \cos(2\alpha_k)x+\sin(2\alpha_k)y,\\
 \sin(2\alpha_k)x-\cos(2\alpha_k)y, z)\in S^3_1.
\end{multline*}
\else
\[
 S^3_1 \ni
 (t,x,y,z)\longmapsto
 (t, 
 \cos(2\alpha_k)x+\sin(2\alpha_k)y,
 \sin(2\alpha_k)x-\cos(2\alpha_k)y, z)\in S^3_1.
\]
\fi
This involution fixes the two lines
$L^{+}_k$ and $L^{-}_k$.
Suppose that $m$ is an odd integer.
By \ref{item:S0:1} of Proposition~\ref{prop:sing-end},
$\gamma_m$ is $\pi$-periodic.
In this case, one half of the hyperbola
at $\theta+\pi$ is just the other half
of the hyperbola 
\eqref{eq:hyp0}
at $\theta$, and 
$\CC^{\II}_m$ coincides with the closure
of the image of $f^{\II}_m$.

In the case
$m$ is even,
$\CC^{\II}_m$ 
does not coincide with the closure
of the image of $f^{\II}_m$.
Moreover, $\CC^{\II}_m$ contains 
the image of the map $\iota\circ f^{\II}_m$,
which is
congruent to $f^{\II}_m$,
where $\iota$ is the involution as in \eqref{eq:involution},
and $\CC^{\II}_m$ is just the closure
of the union of the images of $f^{\II}_m$
and $\tilde f^{\II}_m$.
Figure~\ref{fig:h2}
shows $\CC^{\II}_m$
and
the image of $f^{\II}_m$ for $m=2$.

Summarizing the above,
we get the following:

\begin{thm}\label{thm:B}
 For each $m=2,3,\dots$,
 the set $\CC^{\II}_m$
 gives the real analytic extension of
 the exceptional catenoid
 $f^{\II}_m$, and 
 has the following properties:
\begingroup
\renewcommand{\theenumi}{(\roman{enumi})}
\renewcommand{\labelenumi}{(\roman{enumi})}
 \begin{enumerate}
  \item\label{item:B:1} 
       The projection of 
       $\CC^{\II}_m$
       into the $xy$-plane in $\R^4_1$
       is the hypo-trochoid $-2\gamma_m$.
       Furthermore,
       the section of $\CC^{\II}_m$ by
       a plane containing 
       a point of the hypo-trochoid and
       perpendicular to the $xy$-plane
       is a hyperbola
       unless 
       the plane passes
       through the cone-like singularity 
       of $\CC^{\II}_m$.
  \item\label{item:B:2} 
       $\CC^{\II}_m$ is almost immersed
       and has four endpoints. Two of them
       lie in $\partial_+ S^3_1$ and the others
       lie in $\partial_- S^3_1$.
       Moreover,
       $\CC^{\II}_m$ is almost embedded
       if $m=2,3$.
  \item If $m$ is odd, then
	$\CC^{\II}_m$ is the closure of the
	image of  $f^{\II}_m$. 
	On the other hand, if $m$ is even,
	then the closure of the image of 
	$f^{\II}_m$
	is just half of $\CC^{\II}_m$.
	The other half can be obtained by 
	the isometric involution of $S^3_1$
	given in \eqref{eq:involution}.
 \end{enumerate}
\endgroup
\end{thm}

When $m\ge 4$, $\CC^{\II}_m$
has self-intersections.
It should be remarked that
similar phenomena occur
for parabolic or hyperbolic catenoids
in the class of
space-like maximal surfaces
in $\R^3_1$ (see \cite{FKKRSUYY2}).

\begin{rem}
As shown in \cite{FKKRUYY2}, 
$\CC^{\II}_m$ is analytically complete,
that is, $\CC^{\II}_m$ admits no analytic extension. 
\end{rem}

To end this paper, we remark that
the replacement 
\[
   s\mapsto \imag s\qquad (r=e^s)
\]
of the parameter of $f^{\II}_m$ induces constant
mean curvature surfaces in 
anti-de Sitter space.
This induces a
family of surfaces
\[
  \check f_m:=\bigl(x_0(s,\theta),x_1(\theta),x_2(\theta),
                  x_3(s,\theta)\bigr)
\]
given by
$(x_0,x_3)= \tfrac{1-m^2}{2m}\cos m\theta
   \bigl(\cos s,\sin s\bigr)$,
and $x_1,x_2$ as in \eqref{eq:parametrization},
where $m=2,3,4,\dots$.
For each $m$, the corresponding
surface lies in the space form 
\[
  H^3_1(-1):=\{
    (t,x,y,z)\,;\,
       t^2-x^2-y^2+z^2=-1\}
\]
of constant curvature $-1$ realized in $(\R^4_2,+--+)$.
The image of $\check f_m$ gives a compact 
almost immersed time-like surface
of constant mean curvature one
having a finite number of cone-like singularities.
Moreover, if
$m$ equals $2$ or $3$, 
the surface is almost embedded 
in the sense given in the introduction.
To draw the surfaces, we use the
`solid torus model' of $H^3_1(-1)$, that is,
we define the following projection
\ifx\distribute\undefined
\begin{multline*}
 \check \varPi:H^3_1(-1)\ni (t,x,y,z)\longmapsto \\
\ifx\undefined\distribute
 \text{\small$
\fi
 \frac{1}{\rho}\left(
 \left(1 + \frac{t}{\rho}\right)x,
 \left(1 + \frac{t}{\rho}\right)y, z\right)
\ifx\undefined\distribute
$}
\fi\in \R^3,
\end{multline*}
\else
\[
 \check \varPi:H^3_1(-1)\ni (t,x,y,z)\longmapsto 
 \frac{1}{\rho}\left(
 \left(1 + \frac{t}{\rho}\right)x,
 \left(1 + \frac{t}{\rho}\right)y, z\right)\in \R^3,
\]
\fi
where
$\rho:=\sqrt{x^2+y^2}$.
The image of $\check \varPi$
is the interior of the solid
torus obtained by
rotating the unit disk with center $(1,0,0)$
about  the third axis
in $\R^3$.
The images of $\check \varPi\circ \check f_m$
for $m=2,3$ are given in Figure~\ref{fig:anti}.
\ifx\undefined\distribute
\begin{figure}[h!]%
\else
\begin{figure}
\fi
 \begin{center}
       \includegraphics[width=3.3cm]{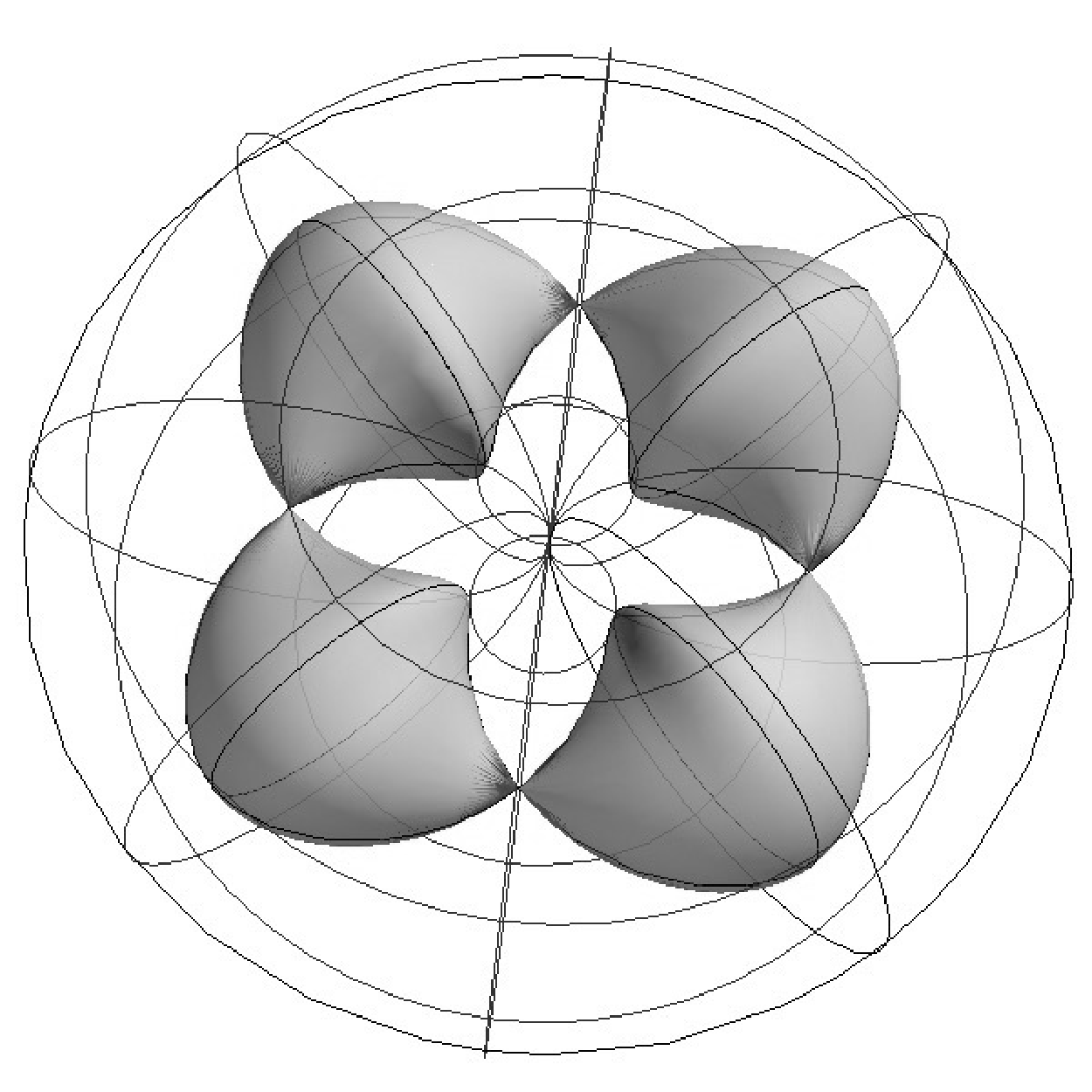} 
\ifx\undefined\distribute
\else
  \hspace{1.5cm}
\fi
       \includegraphics[width=3.3cm]{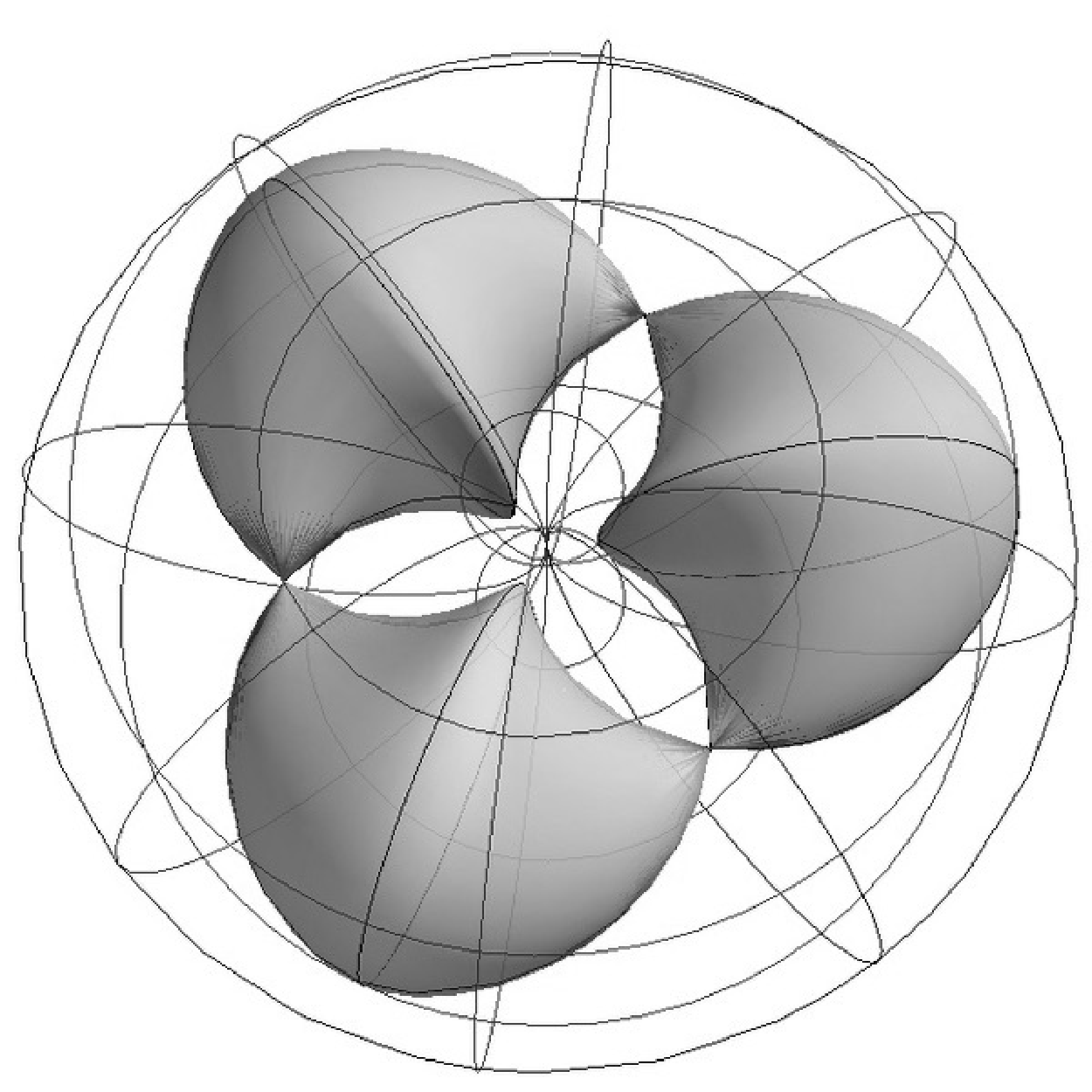} 
 \end{center}
 \caption{%
 The images of $\check f_m$ for $m=2$ (left)
 and $m=3$ (right).
 }%
\label{fig:anti}
\end{figure}


\begin{ack}
 Fujimori was partially supported by the Grant-in-Aid
 for Scientific Research (C) No.\ 17K05219,
 Kawakami by (C) No.\ 15K04840,
 Kokubu by  (C) No.\ 17K05227,
 Rossman by (C) No.\ 15K04845,
 Umehara  by (A) No.\ 26247005,
 and Yamada by (B) No.\ 17H0282.
\end{ack}

\end{document}